\documentclass[final,onefignum,onetabnum]{siamart171218}

\usepackage{amssymb}
\usepackage{appendix}

\usepackage{xspace}
\usepackage{bold-extra}
\usepackage[most]{tcolorbox}

\colorlet{texcscolor}{blue!50!black}
\colorlet{texemcolor}{red!70!black}
\colorlet{texpreamble}{red!70!black}
\colorlet{codebackground}{black!25!white!25}

\newtheorem{remark}[theorem]{Remark}

\usepackage{amsmath}

\def\bk{{\bf{k}}}
\def\bi{{\bf{i}}}
\def\bj{{\bf{j}}}

\lstdefinestyle{siamlatex}
{%
  style=tcblatex,
  texcsstyle=*\color{texcscolor},
  texcsstyle=[2]\color{texemcolor},
  keywordstyle=[2]\color{texemcolor},
  moretexcs={cref,Cref,maketitle,mathcal,text,headers,email,url},
}

\tcbset{%
  colframe=black!75!white!75,
  coltitle=white,
  colback=codebackground, 
  colbacklower=white, 
  fonttitle=\bfseries,
  arc=0pt,outer arc=0pt,
  top=1pt,bottom=1pt,left=1mm,right=1mm,middle=1mm,boxsep=1mm,
  leftrule=0.3mm,rightrule=0.3mm,toprule=0.3mm,bottomrule=0.3mm,
  listing options={style=siamlatex}
}

\newtcblisting[use counter=example]{example}[2][]{%
  title={Example~\thetcbcounter: #2},#1}

\newtcbinputlisting[use counter=example]{\examplefile}[3][]{%
  title={Example~\thetcbcounter: #2},listing file={#3},#1}

\DeclareTotalTCBox{\code}{ v O{} }
{ 
  fontupper=\ttfamily\color{black},
  nobeforeafter,
  tcbox raise base,
  colback=codebackground,colframe=white,
  top=0pt,bottom=0pt,left=0mm,right=0mm,
  leftrule=0pt,rightrule=0pt,toprule=0mm,bottomrule=0mm,
  boxsep=0.5mm,
  #2}{#1}

\patchcmd\newpage{\vfil}{}{}{}
\begin{tcbverbatimwrite}{tmp_\jobname_header.tex}
\title{Hypocoercivity Based Sensitivity Analysis and Spectral Convergence of the Stochastic Galerkin Approximation to Collisional Kinetic Equations with Multiple Scales and Random Inputs 
\thanks {Submitted to the editors April 6, 2018}
\funding{{Both authors were partially supported by NSF grants
DMS-1522184 and DMS-1107291: RNMS KI-Net,
and by the Office of the Vice Chancellor for Research and
Graduate Education at the University of Wisconsin-Madison with funding from the Wisconsin
Alumni Research Foundation. The first author was also partially supported by DOE--Simulation Center for Runaway Electron Avoidance and Mitigation. 
}}}

\author{Liu Liu\thanks{The Institute for Computational Engineering and Sciences (ICES), University of Texas at Austin, Austin, Texas 78705, USA (\email{lliu@ices.utexas.edu}).}
\and Shi Jin\thanks{Department of Mathematics, University of Wisconsin-Madison, Madison, WI 53706, USA
(\email{sjin@wisc.edu}). }}

\headers{SENSITIVITY ANALYSIS FOR KINETIC EQUATIONS}{LIU LIU AND SHI JIN}
\end{tcbverbatimwrite}
\input{tmp_\jobname_header.tex}

\begin{document}
\maketitle
\begin{tcbverbatimwrite}{tmp_\jobname_abstract.tex}
\begin{abstract}
  In this paper we provide a general framework to study general class of linear
  and nonlinear kinetic equations with random uncertainties from the initial
  data or collision kernels, and their stochastic Galerkin approximations,
  in both incompressible Navier-Stokes and Euler (acoustic) regimes.
  First, we show that the general framework put forth in [C. Mouhot and
  L. Neumann, {\it Nonlinearity}, 19, 969-998, 2006; 
  M. Briant, {\it J. Diff. Eqn.},
  259, 6072-6141, 2005]
 based on hypocoercivity for the deterministic kinetic
  equations can be easily adopted for sensitivity
  analysis for random kinetic equations, which gives rise to an exponential
  convergence of the random solution toward the (deterministic) global equilibrium, under suitable conditions on the collision kernel.
  Then we use such theory to study the stochastic Galerkin (SG) methods
  for the equations, establish hypocoercivity of the SG system and regularity of its solution,
  and spectral accuracy and exponential decay in time of the
  numerical error of the method in a weighted Sobolev norm.
\end{abstract}


\end{tcbverbatimwrite}
\input{tmp_\jobname_abstract.tex}

\section{Introduction}
 Consider the initial value problem for kinetic equations of the form
 \begin{align}
\label{model}
\left\{
\begin{array}{l}
  \displaystyle \partial_t f + \frac{1}{\epsilon^\alpha} v\cdot\nabla_x f=
  \frac{1}{\epsilon^{1+\alpha}} \mathcal Q(f),  \\[4pt]
\displaystyle  f(0,x,v,z)=f_{in}(x,v,z), \qquad x\in\Omega\subset\mathbb T^d,  \, v\in \mathbb R^d, z\in I_z \subset \mathbb R,
\end{array}\right.
\end{align}
where $f(t,x,v,z)$ is the distribution of particles in the phase space depending
on time $t$, particle position $x$, velocity $v$ and a random variable
$z$, and $d\geq 1$ denotes the dimension of the spatial and velocity spaces.
$z$ is a random variable that lies in domain $I_z\subset \mathbb R $.
The operator $\mathcal Q$ models the collisional interactions of particles, which is either binary or between particles against a
surrounding medium. $\epsilon$ is the Knudsen number, the dimensionless
ratio of particle mean free path over the domain size. $\alpha=1$ is referred
to the incompressible Navier-Stokes scaling, while $\alpha=0$ corresponds
to the Euler (or acoustic in this article) scaling.
The periodic boundary condition for the spatial domain $\Omega=\mathbb T^d$ is assumed here.

The main goal of this paper is to study the above kinetic equation
and its numerical approximation
under the influence of {\it random uncertainty}.
Since kinetic equations are not  first-principle physical equations,
there are inevitably modeling errors,
incomplete knowledge of the interaction mechanism, and imprecise measurement of the initial and boundary data,
which contribute uncertainties to the equations. Understanding the impact of these uncertainties
is crucial to the simulation and validation of the models, in order to
provide more reliable predictions, calibrations and improvements of the models.
In this paper we consider the uncertainty coming from initial data and collision kernels.
The uncertainty is described by the random variable $z$, which lies in the random space
$I_z$ with a probability measure $\pi(z)dz$, then the solution $f=f(t,x,v,z)$ depends on $z$.
The {\it sensitivity analysis} aims  to study how randomness of the initial data and collision kernel (the ``input'') propagates in time and how it affects the solution in the long time (the ``output'') \cite{SmithBook}.
It is an essential part for the so-called {\it uncertainty quantification} for
kinetic equations.

For a general class of linear collisional kinetic models in the torus without
uncertainty variable $z$, including the linearized Boltzmann equation for hard spheres,
the linearized Landau equation with hard and moderately soft potentials and the semi-classical linearized fermionic and bosonic relaxation models,
based on the hypocoercivity theory established by Mouhot and Neumann \cite{CN},
Briant \cite{MB} proved explicit coercivity estimates for some modified Sobolev norms on the associated integro-differential operator. For the full
nonlinear models including the Boltzmann, Landau and semi-classical relaxation model of quantum Boltzmann equation,  \cite{MB} deduced the existence of classical solutions near the global equilibrium and obtained explicit estimates on the exponential
convergence rate towards equilibrium. We first show that this general hypocoecivity theory can be easily adopted for the uncertain kinetic equation (\ref{model}) to obtain a similar theory of convergence to the (deterministic) global
equilibrium, in a weighted Sobolev norm including the
random space. For the case of random initial data, the analysis
is basically the same as those in \cite{MB} except one has to check that
the estimate constants are independent of $z$ and the estimates need to be extended
for the high-order derivatives in $z$.  
When the collision kernel is random with bounded $z$-derivatives, 
for the Boltzmann equation, 
while the nonlinear portion of the collision operator needs
a slight generation to include the high order derivative of the collision kernel in $z$, 
we adopt the ideas in \cite{Ma, Liu} with some new estimates for the linearized collision operator. 
The results
show that the impact of the random uncertainty will
diminish in time, namely the long time solution is insensitive to the random
perturbation of the initial data and the collision kernel, for both the
incompressible Navier-Stokes and acoustic scalings.

To numerically solve such equations with uncertainties, one of the standard
and efficient numerical methods is the generalized polynomial chaos approach in the stochastic Galerkin (referred as gPC-SG)
framework \cite{Ghanem, GWZ, Xiu, XZJ, Hu, HuReview}.
Compared with the classical Monte-Carlo method, the gPC-SG approach enjoys a spectral accuracy in the random space--
if the solution is sufficiently smooth--while the Monte-Carlo method converges with only half-th order accuracy.
In the second part of this paper, by extending the hypocoercivity
analysis to the gPC-SG system, using a weighted norm first introduced by Shu
and Jin \cite{Rui}, we prove the exponential decay in time toward the global equilibrium
and the spectral accuracy of the gPC-SG approximation in both incompressible
Navier-Stokes and acoustic scalings,
under the assumption of small $\mathcal O(\epsilon)$ random perturbation to the collision kernel and
boundedness of the random domain $I_z$, with some additional assumption
for the orthogonal polynomials used in the gPC approximation.

While uncertainty quantification has been a hot topic in scientific and engineering computing in the last two decades,
research on uncertainty quantification for kinetic equations has been relatively recent. We refer to recent review articles \cite{HuReview, DPZ} and
some recent works \cite{XZJ, Hu, DesPer, Jin-Liu, Ma, Mu-Liu, QinWang, Liu, Jin-Lu, Zhu, Rui1, Rui, APZ} in this direction. The first sensitivity analysis similar
to this paper for the linear
transport equation, with uniform (in the Knudsen number) spectral convergence
of the gPC-SG approximation, was given by Jin, Liu and Ma in \cite{Ma}.
For similar theory for linear transport equation with anisotropic
collision kernel, see \cite{Jin-Liu, Liu}. Uniform regularity for general
linear transport equations conserving mass, based on hypocoercivity
established in \cite{DMS},
was obtained by Li and Wang in \cite{QinWang}. The first regularity
for a nonlinear kinetic equation, the Vlasov-Poisson-Fokker-Planck system
with random initial data in both high-field and parabolic regimes, was established by Jin and Zhu \cite{Zhu}.
Shu and Jin obtained uniform regularity and spectral convergence of the gPC-SG
system to a nonlinear Fokker-Planck-incompressible Navier-Stokes
system with uncertain initial data in \cite{Rui}.  In this paper, not only
were our results new on the Boltzmann equation with uncertainties in both
the continuous and the discrete gPC-SG equations, we also give
a {\it unified} approach,
for both incompressible Navier-Stokes (diffusive) and acoustic
(Euler) scalings, which applies to a wide class of  both linear and nonlinear
kinetic (such as Boltzmann, Landau, the relaxation model of quantum Boltzmann)
equations with uncertainties in initial data and collision kernels.

As we finish this manuscript, Zhu announced similar convergence results for the Boltzmann equation with random initial data and the Euler scaling,
and the stability and regularity of its stochastic Galerkin approximation
\cite{Zhu-B}, using techniques specific to the Boltzmann equation in the whole-space case \cite{Guo-whole, Duan}. 

This paper is organized as follows. Section \ref{Section1} provides
the theoretical framework and  hypocoercivity assumptions for general kinetic models, and the results on exponential decay to the global equilibrium. The proof of some of the convergence results are given in section \ref{Proof_Thms}. 
In section 4, we prove that the theoretical results of
section \ref{Section1} are valid for the Boltzmann equation with both
random initial data and random collision kernel. 
Section \ref{gPC} proves the hypocoercivity, and exponential time decay of the solution of the gPC-SG approximation to the uncertain Boltzmann equation,
with numerical  error shown to be spectrally accurate and exponentially decaying in time. The paper is concluded in section \ref{conclusion}.

\section{General Framework and Convergence to the Global Equilibrium}
\label{Section1}

In this section, we describe the abstract framework and assumptions of the hypocoercivity theory, introduced in \cite{CN, MB}, extend them to include the
random dependence, and then give the results about
convergence toward global equilibrium for the nonlinear kinetic equations
with uncertainty.
The results are stated for the case of random initial data. The case of random collision kernel can be included
in the same framework. See subsection \ref{kernel} for the Boltzmann equation.

\subsection{Theoretical Framework: Perturbative Setting and Small Scalings}
In the sequel $\mathcal L$ is used for both the linear models and the linearized models for nonlinear equations such as Boltzmann, Landau or semi-classical relaxation models, etc.
Suppose $g\in L^2(\Omega\times{\mathbb R^d})$ solves the linear kinetic equation
\begin{equation}\label{linear}\partial_t g + \frac{1}{\epsilon^\alpha} v\cdot\nabla_x g=\frac{1}{\epsilon^{1+\alpha}} \mathcal L(g)\,, \end{equation}
where $\mathcal L$ is a linearized collision operator depending on the precise form of the collision operator $\mathcal Q$.
As summarized in \cite{MM}, the idea is to employ the hypocoercivity of the linearized operator
\[\mathcal G=\frac{1}{\epsilon^{1+\alpha}}\mathcal L-\frac{1}{\epsilon^{\alpha}}
  \mathcal T\,,
\]
where
$\mathcal T=v\cdot\nabla_x$  is the streaming operator, using the dissipative properties of $\mathcal L$ and the conservative properties of $\mathcal T$.
The aim is to find a functional $\eta [g]$ which is equivalent to the square of the norm of a Banach space,
for example
\[H_{x,v}^1=\{g\, |\int_{\Omega\times\mathbb R^d}\, \sum_{|i|+|j|\leq 1} ||\partial_{x_i} \partial_{v_j} g ||_{L^2_{x,v}}^2\, dxdv<\infty
  \}, 
\]
in which $L_{x,v}^2 = \{ g\, | \int_{\Omega\times\mathbb R^d} \, g^2 \, dxdv < \infty\}$, such that $$\kappa_1 ||g||_{H_{x,v}^1} \leq \eta[g] \leq \kappa_2 ||g||_{H_{x,v}^1}, \qquad \text{for  } g\in H_{x,v}^1, $$
which leads to $$\frac{d}{dt}\, \eta[g(t)] \leq -\kappa ||g(t)||_{H_{x,v}^1}, \qquad t>0, $$
with constants $\kappa_1$, $\kappa_2$, $\kappa>0$. Then one concludes the exponential convergence of $g$ in $H^1_{x,v}$.
The obvious choice of $\eta[g]= c_1\, ||g||_{L^2_{x,v}}^2 + c_2\, ||\nabla_x g||_{L^2_{x,v}}^2 + c_3\, ||\nabla_v g||_{L^2_{x,v}}$ does not work.
The key idea, first seen in \cite{Villani} and implemented in \cite{CN}, is to add the ``mixing term"
$c\langle\nabla_x g, \, \nabla_v g\rangle_{L^2_{x,v}}$ to the definition of $\eta[g]$, that is
$$\frac{d}{dt}\langle\nabla_x g, \, \nabla_v g\rangle_{L^2_{x,v}} =
-||\nabla_x g||_{L^2_{x,v}}^2 + 2\langle\nabla_x\mathcal L(g), \, \nabla_v g\rangle_{L^2_{x,v}}\,.$$

It was proved in \cite{CN} that if the linear operator $\mathcal L$ satisfies some assumptions, then 
$\mathcal L -v\cdot\nabla_x$ generates a strongly continuous evolution semi-group $e^{t\mathcal{G}}$ on $H_{x,v}^s$, which satisfies
\begin{equation}||e^{t\mathcal{G}}(\mathbb I-\Pi_{\mathcal G})||_{H^s_{x,v}} \leq C \exp[-\tau t], \label{CN_thm1} \end{equation}
for some explicit constants $C, \, \tau>0$ depending only on the constants determined by the equation itself.
This result shows that apart from $0$, the spectrum of $\mathcal G$ is included in
$$\{\xi \in\mathbb C: \text{Re}(\xi)\leq -\tau\}. $$

{\bf The Perturbative setting: }
Equations defined in (\ref{model}) admit a unique global equilibrium in the torus, denoted by $\mathcal M$ which is independent of $t,x$.
Now consider the linearization around this equilibrium and perturbations of the solution of the form
\begin{equation}
  f=\mathcal M+ \epsilon M h\,, 
  \label{per-f}
  \end{equation}
with $$\mathcal M=\frac{1}{(2\pi)^{\frac{d}{2}}}\, e^{-\frac{|v|^2}{2}}\,,$$
and $M=\sqrt{\mathcal M}$. 
The linear (or linearized) operator $\mathcal L$ is acting on $L^2_v=\{f\, |\int_{\mathbb R^d} f^2\, dv <\infty\}$, 
with the kernel denoted by $N(\mathcal L)=\text{Span}\{\varphi_1, \cdots, \varphi_n\}$.
$\{\varphi_i\}_{1\leq i\leq n}$ is an orthonormal family of polynomials in $v$ corresponding to the manifold of local equilibria for the linearized kinetic models.
The orthogonal projection on $N(\mathcal L)$ in $L^2_v$ is defined by
\begin{equation}\label{Pi}\Pi_{\mathcal L} (h) = \sum_{i=1}^{n}\, \left(\int_{\mathbb R^d} h\varphi_i\, dv\right)\varphi_i, \end{equation}
where $\Pi_{\mathcal L}$ is the projection on the `fluid part' and $\mathbb I-\Pi_{\mathcal L}$ is the projection on the kinetic part,
with $\mathbb I$ the identity operator.
 The global equilibrium is then
\begin{equation}\label{PiG}\Pi_{\mathcal G} (h) =\sum_{i=1}^{n}\, \left(\int_{\mathbb T^d\times{\mathbb R^d}}h\varphi_i\, dx dv\right)\varphi_i,
\end{equation}
which is independent of $x$ and $t$ and is the orthogonal projection on $L^2_{x,v}$. \\[1pt]

{\bf Small scalings and main idea for the full equation: }
With the small scaling $\epsilon$, the problem becomes more interesting and challenging.
The project was initiated by Bardos, Golse and Levermore \cite{Bardos1, Bardos2} to derive the fluid limits which include incompressible
Navier-Stokes, compressible Euler equations and acoustic system from the DiPerna-Lions renormalized solutions \cite{Lions}.
See for example \cite{GSR, LevMas}. Here we will study the solution
in the perturbative setting (\ref{per-f}), which gaurantees that the solution
will be classical, thus allows one to conduct estimates in the Sobolev space
\cite{Guo-NS}.
\cite{MB} considers the kinetic  equation (\ref{model}) with the incompressible Navier-Stokes scaling ($\alpha=1$). With (\ref{per-f}),  $h$ satisfies
\begin{equation}\partial_t h + \frac{1}{\epsilon}v\cdot\nabla_x h=\frac{1}{\epsilon^2}\mathcal L(h)+\frac{1}{\epsilon}\mathcal F(h,h)\,. 
  \label{INS-scaling}
\end{equation}

Due to the small scaling, if one directly applies the estimates in \cite{CN}, typically the $v$-derivatives contribute to the energy norm by
a factor of $1/\epsilon$.
This prevents one from having a uniform exponential decay for the $v$-derivatives.
As initiated by Guo \cite{Guo-NS}, one needs to study the $v$ derivatives of the microscopic part of
the solution $h$. This allows \cite{MB} to construct a new energy norm to capture the structure of $\mathcal L$ on its orthogonal part, which, when combined with the previous
strategy, leads to a uniform exponential decay for solutions close to the global equilibrium.
The result is uniform in $\epsilon$, thus gives a strong convergence in time to the incompressible Navier-Stokes equations as $\epsilon$ goes to zero,
under some assumptions on the initial conditions.
\cite{MB} also gives the proof of existence of solutions close to the global equilibrium.

Another important scaling is the compressible Euler (or acoustic) scaling
($\alpha=0$), in which $h$ solves
\begin{equation}
  \label{Euler-scaling}
  \partial_t h + v\cdot\nabla_x h=\frac{1}{\epsilon}\mathcal L(h)+\mathcal F(h,h)\,. 
  \end{equation}
The authors in \cite{Jang09, LevMas, Levermore, Jang-Guo1,AMUXY} studied the acoustic limit of the Boltzmann equation in the framework of classical solutions of the form (\ref{per-f}).
They established the global-in-time,
uniform-in-$\epsilon$ energy estimates for the perturbated solution $h$ and proved its strong convergence to the distribution function whose dynamics
is governed by the acoustic system, which is the linearization of the homogeneous state of the compressible Euler system.
Furthermore, \cite{Jang-Guo2} studied the compressible Euler limit of the Boltzmann equation by using the local Hilbert expansion around the local
equilibrium for smooth solutions, which was first done by Caflisch in \cite{Caflisch}.

In this paper, we will focus on the incompressible Navier-Stokes and the
acoustic scaling for solutions of the form (\ref{per-f}).

\subsection{Hypocoercivity Assumptions}
\label{sec1}
We first discuss the following assumptions. 

{\bf Assumptions on the linear operator $\mathcal L$ in $H^1_{x,v}$: } \\
{\bf H1}. $\mathcal L: L^2=L^2(\mathbb T^d \times{\mathbb R^d})$ is closed, self-adjoint on $L^2_v$ and local in $t,x$.
$\mathcal L$ has the form $\mathcal L=K-\Lambda$. There is a norm $||\cdot||_{\Lambda_v}$ on $\mathbb R^d$, such that $\forall\,  h\in L^2_v$,
$\Lambda$ satisfies the {\it coercivity condition}: \begin{equation}\label{H1_1}\nu_0^{\Lambda}\,  ||h||_{L^2_v}^2 \leq \nu_1^{\Lambda}\, ||h||_{\Lambda_v}^2 \leq \langle \Lambda(h),\,h\rangle_{L^2_v}\leq \nu_2^{\Lambda}\,  ||h||_{\Lambda_v}^2\,, \end{equation}
and $\forall\, h\in H^1_{v}$,
\begin{equation}\label{H1}\langle \nabla_v \Lambda(h), \,\nabla_v h\rangle_{L^2_v} \geq \nu_3^{\Lambda}\, ||\nabla_v h||_{\Lambda_v}^2 - \nu_4^{\Lambda}\, ||h||_{L^2_v}^2\,,  \end{equation}
where $(\nu_s^{\Lambda})_{1\leq s\leq 4}>0$ are constants depending on the operators and the velocity space.
One further assumes that $\forall\, h, g\in L^2_v$,
\begin{equation}\label{LL}\langle\mathcal L(h), \, g\rangle_{L^2_v}\leq C^{\mathcal L}\, ||h||_{\Lambda_v}||g||_{\Lambda_v}\,.\end{equation}
{\bf H2}. $K$ has a regularizing effect. $\forall\, \delta>0$, there exists some explicit constant $C(\delta)>0$ such that $\forall\, h\in H_v^1$,
\begin{equation}\label{H2}\langle\nabla_v K(h), \, \nabla_v h\rangle_{L^2_v} \leq C(\delta)\, ||h||_{L^2_v}^2 + \delta\, ||\nabla_v h||_{L^2_v}^2\,.
\end{equation}
{\bf H3}. $\mathcal L$ has a finite dimensional kernel $$N(\mathcal L)=\text{Span}\{\varphi_1, \cdots, \varphi_n\}.$$
$\Pi_{\mathcal L}(h)$ given in (\ref{Pi}) is the orthogonal projection in $L^2_v$ on $N(\mathcal L)$. $\mathcal L$ has the {\it local coercivity property}:
There exists $\lambda>0$ such that $\forall\, h\in L^2_v$,
\begin{equation}\label{coercivity}\langle\mathcal L(h),\, h\rangle_{L^2_v} \leq -\lambda\, ||h^{\perp}||_{\Lambda_v}^2\,, \end{equation}
where $$h^{\perp}=h-\Pi_{\mathcal L}(h)$$ stands for the {\it microscopic part} of $h$, which satisfies $h^{\perp}\in N(\mathcal L)^{\perp}$
in $L_v^2$.

To extend to higher-order Sobolev spaces, let us first introduce some notations of multi-indices and Sobolev norms.
For two multi-indices $j$ and $l$ in $\mathbb N^{d}$, define $$\partial_l^j = \partial/\partial v_j\, \partial/\partial x_l. $$
For $i\in\{1, \cdots, d\}$, denote by $c_i(j)$ the value of the $i$-th coordinate of $j$ and by $|j|$ the $l^1$ norm of the multi-index, that is,
$|j|=\sum_{i=1}^d c_i(j)$. Define the multi-index $\delta_{i_0}$ by: $c_i(\delta_{i_0})=1$ if $i=i_0$ and $0$ otherwise.
We use the notation $$\partial_z^{\alpha} h = \partial^{\alpha}h. $$
Denote $||\cdot||_{\Lambda}:= ||\, ||\cdot||_{\Lambda_v}\, ||_{L^2_x}$.
The Sobolev norms on $H^{s}_{x,v}$ and $H_{\Lambda}^s$ are defined by
$$||h||_{H_{x,v}^s}^2 = \sum_{|j|+|l|\leq s}\, ||\partial_l^j h||_{L^2_{x,v}}^2\,,\qquad
||h||_{H_{\Lambda}^s}^2 = \sum_{|j|+|l|\leq s}\, ||\partial_l^j h||_{\Lambda}^2\,.$$
Define the sum of Sobolev norms of the $z$ derivatives by
\begin{align*}
&\displaystyle  ||h||_{H_{x,v}^{s,r}}^2 = \sum_{|m|\leq r}\, ||\partial^m h||_{H_{x,v}^s}^2\,, \qquad\qquad
||h||_{H_{\Lambda}^{s,r}}^2 = \sum_{|m|\leq r}\, ||\partial^m h||_{H_{\Lambda}^s}^2\,, \\[2pt]
&\displaystyle ||h||_{H_{x}^{s,r}L_v^2}^2 = \sum_{|m|\leq r}\, ||\partial^m h||_{H_{x}^s L_v^2}^2\,.
\end{align*}
Note that these norms are all functions of $z$. Define the norms in the $(x,v,z)$ space
$$||h||_{H^{s}_{x,v}H_z^r}^2 = \int_{I_z}\, ||h||_{H_{x,v}^{s,r}}^2\, \pi(z)\, dz\,,$$
in addition to the $\sup$ norm in $z$ variable,
\begin{equation}\label{h_sup} ||h||_{H_{x,v}^{s,r} L_z^{\infty}}=\sup_{z\in I_z}\, ||h||_{H_{x,v}^{s,r}}\,, \qquad
||h||_{H_{x,v}^s L_z^{\infty}}=\sup_{z\in I_z}\, ||h||_{H_{x,v}^s}\,. 
\end{equation}

{\bf Assumptions on the linear operator $\mathcal L$ in $H^s_{x,v}$,\, $s>1$: }\\
$\bf {H1^{\prime}}$. For all $s\geq 1$, $|j|+|l|=s$ such that $|j|\geq 1$,
$$\forall h\in H^{s}, \qquad \langle \partial_l^j \Lambda(h),\, \partial_l^j h\rangle_{L^2_{x,v}} \geq \nu_5^{\Lambda}\,
||\partial_l^j h||_{\Lambda}^2 - \nu_6^{\Lambda}\, ||h||_{H^{s-1}_{x,v}}^2\,. $$
$\bf {H2^{\prime}}$. For all $s\geq 1$, for all $|j|+|l|=s$ such that $|j|\geq 1$ and for any $\delta>0$, there exists an explicit $C(\delta)$ such that
$\forall h\in H_{x,v}^s$,
$$\langle\partial_l^j K(h), \, \partial_l^j h\rangle_{L^2_{x,v}} \leq C(\delta)\, ||h||_{H^{s-1}_{x,v}}^2 + \delta ||\partial_l^j h||_{L^2_{x,v}}^2\,.$$
{\bf H4. Orthogonality to $N(\mathcal L)$:}
$$\forall\, h, g \in \text{Dom}(\mathcal F)\cap L_v^2, \qquad \mathcal F(g,h)\in N(\mathcal L)^{\perp}, $$
where $\text{Dom}(\mathcal F)$ stands for the domain of the operator $\mathcal F$.

Due to the uncertainties introduced to the system, we also make the following assumption on the nonlinear term,
which is slightly different from \cite{MB}.  \\[1pt]

{\bf Assumptions on the nonlinear term $\mathcal F$: }\\
{\bf H5}. $\mathcal F: L^2_v \times L^2_v \to L^2_v$ is a bilinear symmetric operator such that for
all multi-indexes $j$ and $l$ such that $|j|+|l|\leq s$, $s\geq 0$, $m\geq 0$,
\begin{align*}
\left| \langle\partial^{m}\partial_l^j \mathcal F(h,h), \, f\rangle_{L^2_{x,v}}\right| \leq
\begin{cases}
&\displaystyle\mathcal G_{x,v,z}^{s,m}(h,h)\, ||f||_{\Lambda}\,, \qquad \text{if   }  j \neq 0, \\[2pt]
&\displaystyle \mathcal G_{x,z}^{s,m}(h,h)\, ||f||_{\Lambda}\,,  \qquad\text{if   } j=0.
\end{cases}
\end{align*}
Sum up $m=0, \cdots, r$,
then $\exists\, s_0\in \mathbb N, \, \forall s\geq s_0$, there exists a $z$-independent $C_{\mathcal F}>0$ such that for all $z$,
\begin{align*}
&\displaystyle\sum_{|m|\leq r}\, (\mathcal G_{x,v,z}^{s,m}(h,h))^2  \leq C_{\mathcal F}\, ||h||_{H^{s,r}_{x,v}}^2\, ||h||_{H_{\Lambda}^{s,r}}^2, \\[4pt]
&\displaystyle\sum_{|m|\leq r}\, (\mathcal G_{x,z}^{s,m}(h,h))^2 \leq C_{\mathcal F}\, ||h||_{H_{x}^{s,r}L^2_{v}}^2\, ||h||_{H_{\Lambda}^{s,r}}^2\,.
\end{align*}

\subsection{Convergence to the Global Equilibrium}
\label{Thms}
Define a positive functional on $H_{x,v}^{s}$, with a dependence on $\epsilon$,
$$||\cdot||_{\mathcal H_{\epsilon}^s}^2 = \sum_{|j|+|l|\leq s, |j|\geq 1} \epsilon^2\, b_{j,l}^{(s)}\, ||\partial_l^{j}\cdot||_{L^2_{x,v}}^2
+ \sum_{|l|\leq s}\alpha_l^{(s)}\, ||\partial_l^{0}\cdot||_{L^2_{x,v}}^2 + \sum_{|l|\leq s,\, i, c_i(l)>0} \epsilon\, a_{i,l}^{(s)}\,
\langle\partial_{l-\delta_i}^{\delta_i}\cdot, \,\partial_l^{0}\cdot\rangle_{L^2_{x,v}}\,, $$
and the Sobolev norms
$$ ||h||_{\mathcal H_{\epsilon}^{s,r}}^2 = \sum_{|m|\leq r}\, ||\partial^m h||_{\mathcal H_{\epsilon}^s}^2\,, \qquad
||h||_{\mathcal H_{\epsilon}^{s,r}L_z^{\infty}} = \sup_{z\in I_z}||h||_{\mathcal H_{\epsilon}^{s,r}}\,.$$
The proof of the theorems in this section is similar to \cite{MB}, except that we need to estimate the (higher-order) derivatives in $z$.
We consider the perturbed form of the solution (\ref{per-f}), with initial
condition
\begin{equation}
h(0,x,v,z)=h_{in}(x,v,z),
\end{equation}
under the incompressible Navier-Stokes scaling (\ref{INS-scaling}). For results on the acoustic scaling, see Remark \ref{RK2}. 

\begin{theorem}
\label{thm1}
If $g$ is the solution to the linear equation
\begin{equation}\label{GL}\partial_t g + \frac{1}{\epsilon}v\cdot\nabla_x g = \frac{1}{\epsilon^2} \mathcal L(g), \end{equation} then

(1) $\forall\, 0<\epsilon\leq\epsilon_d$,  for some $0<\epsilon_d\leq 1$,
the operator $\mathcal G_{\epsilon}$ defined by \begin{equation}\label{GE}\mathcal G_{\epsilon}(g)= \frac{1}{\epsilon^2}\mathcal L(g) - \frac{1}{\epsilon} v\cdot\nabla_x g \end{equation}
generates a $C^{0}$-semigroup on $H_{x,v}^{s}$.

(2) $\exists\, C_{\mathcal G}^{(s)}, \, (b_{j,l}^{(s)}), \, (\alpha_l^{(s)}), \, (a_{j,l}^{(s)})>0$ such that $\forall\, 0<\epsilon\leq \epsilon_d$,
$$ ||\cdot||_{\mathcal H_{\epsilon}^s}^2 \sim\left( ||\cdot||_{L^2_{x,v}}^2 + \sum_{|l|\leq s}||\partial_l^{0}\cdot||_{L^2_{x,v}}^2
+\epsilon^2 \sum_{|l|+|j|\leq s, |j|\geq 1}||\partial_l^{j}\cdot||_{L_{x,v}}^2\right), $$
and $\forall\, g$ in $H_{x,v}^{s}$ and all $z$,
$$\langle \mathcal G_{\epsilon}(g), \, g\rangle_{\mathcal H_{\epsilon}^s} \leq -C_{\mathcal G}^{(s)}\, ||g -
\Pi_{\mathcal G_{\epsilon}}(g) ||_{H_{\Lambda}^s}^2\,.$$

(3) For solution to the nonlinear equation (\ref{INS-scaling}), 
$\forall\, h_{\text{in}}\in H_{x,v}^{s,r}\cap N(\mathcal G_{\epsilon})^{\perp}$,
$h\in \text{Dom}(\Gamma)\, \cap\,  H_{x,v}^{s}$, $\forall\, m\leq r$ and $s\in \mathbb N$,
then
\begin{equation}\label{Th1_3}\frac{d}{dt}||\partial^{m}h||_{\mathcal H_{\epsilon}^{s}}^2 \leq - K_0^{(s)}\, ||\partial^{m}h||_{H_{\Lambda}^s}^2
+ K_1^{(s)}\left(\mathcal G_{x,z}^{s,m}(h, h)\right)^2 + \epsilon^2\, K_2^{(s)}\left( \mathcal G_{x,v,z}^{s,m}(h, h)\right)^2\,. \end{equation}
Moreover, $\exists\, s_0\in\mathbb N$, such that $\forall s\geq s_0$, (\ref{Th1_3}) leads to
\begin{equation}\label{Th1_4}\frac{d}{dt}||h||_{\mathcal H_{\epsilon}^{s,r}}^2 \leq -K_0^{(s)}\, ||h||_{H_{\Lambda}^{s,r}}^2
+ K_1^{(s)}\, ||h||_{H_x^{s,r} L_v^2}^2\, ||h||_{H_{\Lambda}^{s,r}}^2 + \epsilon^2\, K_2^{(s)}\, ||h||_{H_{x,v}^{s,r}}^2\,  ||h||_{H_{\Lambda}^{s,r}}^2 \,.
\end{equation}

(4) $\forall\, 0<\epsilon\leq \epsilon_d$, for some $0<\epsilon_d<1$ and $\forall\, s\geq s_0$, $\exists\, \delta_s, \, C_s, \, \tau_s>0$, such that:
For any distribution $0\leq f_{\text{in}}\in L^1(\mathbb T^d\times\mathbb R^d\times I_z)$ with $f_{\text{in}}=\mathcal M+\epsilon\, M h_{in}$ and
$h_{in}\in N(G_{\epsilon})^{\perp}$, if $||h_{in}||_{\mathcal H_{\epsilon}^{s,r}L_z^{\infty}}\leq\delta_s$,
then there exists a unique global smooth (in $H_{x,v}^{s,r}$, continuous in time) solution $f=f(t,x,v,z)$ satisfying
$f\geq 0$ with $f =\mathcal M + \epsilon\, M h$,  and
\begin{equation}\label{INS} ||h||_{\mathcal H_{\epsilon}^{s,r}L_z^{\infty}} \leq \delta_s\, e^{-\tau_s t}\,.\end{equation}
Furthermore, \begin{equation}\label{INS1} ||h||_{\mathcal H_{\epsilon}^s H_z^r} \leq \delta_s\, e^{-\tau_s t}\,.\end{equation}
In this Theorem, all constants are independent of $z$.
\end{theorem}
\begin{remark}
\label{RK2}
This theorem gives an exponential decay for the semigroup generated by $\mathcal G_{\epsilon}$ defined in (\ref{GE}).
(3) and (4) of Theorem \ref{thm1} give
$$||h||_{H_{x,v}^{s,r}L_z^{\infty}} \leq \frac{\delta_s}{\epsilon} e^{-\tau_s t}\,,$$
under the incompressible Navier-Stokes scaling. This shows that the $v$-derivatives grow at a rate of $1/\epsilon$.
Combining with the work by Guo \cite{Guo-NS} who studies
the fluid part and the microscopic part of the solution $h^{\perp}$ independently, the author in \cite{MB} constructs a new norm defined by (\ref{norm1}),
builds up a functional that is equivalent to the standard Sobolev norm and obtains an exponential decay in
$H_{x,v}^s$\,.
\end{remark}

With uncertainty in the equation, following a similar framework, we have Proposition \ref{prop} and Theorem \ref{thm2}.
Define $||\cdot||_{\mathcal H_{\epsilon_{\perp}}^s}$ by
\begin{align}
&\displaystyle ||\, \cdot\, ||_{\mathcal H_{\epsilon_{\perp}}^s}^2 = \sum_{|j|+|l|\leq s, \, |j|\geq 1}
 b_{j,l}^{(s)}\, ||\partial_l^{j} (\mathbb I -\Pi_{\mathcal L})\cdot\, ||_{L^2_{x,v}}^2 + \sum_{|l|\leq s}\alpha_l^{(s)}\, ||\partial_l^{0}\cdot\, ||_{L^2_{x,v}}^2 \notag\\[2pt]
&\displaystyle\label{norm1}\qquad\qquad\quad +\sum_{|l|\leq s, \, i, c_i(l)>0}\,  \epsilon\,  a_{i,l}^{(s)}\,\langle\partial_{l-\delta_i}^{\delta_i}\cdot, \, \partial_l^{0}\cdot\, \rangle_{L^2_{x,v}}\,,
\end{align}
and the corresponding Sobolev norms
$$ ||h||_{\mathcal H_{\epsilon_{\perp}}^{s,r}}^2 = \sum_{|m|\leq r}\, ||\partial^m h||_{\mathcal H_{\epsilon_{\perp}}^s}^2\,, \qquad
||h||_{\mathcal H_{\epsilon_{\perp}}^{s,r}L_z^{\infty}}= \sup_{z\in I_z}||h||_{\mathcal H_{\epsilon_{\perp}}^{s,r}}\,. $$
\begin{proposition}
\label{prop}
Let $\mathcal L$ be a linear operator satisfying assumptions ${\bf H1^{\prime}}$, ${\bf H2^{\prime}}$ and
${\bf H3}$ and $\mathcal F$ be a bilinear operator satisfying Assumption ${\bf H5}$.
If $h\in H_{x,v}^{s,r}$ is a solution of (\ref{INS-scaling}), with $h_{in}\in H_{x,v}^{s,r}\cap N(G_{\epsilon})^{\perp}$,
then $\forall\, 0<\epsilon\leq\epsilon_d$, for some $0<\epsilon_d\leq 1$, $\forall\, s\in \mathbb N$ and $m\leq r$,
$\exists\, K_0^{(s)}, \,K_1^{(s)}, \,(b_{j,l}^{(s)}), \,(\alpha_l^{(s)}), \,(a_{i,l}^{(s)})>0$
such that for all $z$, we have
\begin{equation}\label{pr1}\frac{d}{dt}||\partial^{m}h||_{\mathcal H_{\epsilon_{\perp}}^{s}}^2 \leq- \, K_0^{(s)}\left(\frac{1}{\epsilon^2} ||\partial^{m}h^{\perp}||_{H_{\Lambda}^{s}}^2 +
\sum_{1\leq |l|\leq s}||\partial_l^{0}\partial^{m}h||_{L^2_{x,v}}^2 \right)
 + K_1^{(s)}\left(\mathcal G_{x,v,z}^{s,m}(h,h)\right)^2\,. \end{equation}
Furthermore, $\exists\, s_0\in\mathbb N$, $\forall s\geq s_0$, this implies
\begin{equation}
\label{pr2}\frac{d}{dt}||h||_{\mathcal H_{\epsilon_{\perp}}^{s,r}}^2 \leq
- K_0^{(s)}\left(\frac{1}{\epsilon^2} ||h^{\perp}||_{H_{\Lambda}^{s,r}}^2 +
\sum_{|m|\leq r}\sum_{1\leq |l|\leq s}\, ||\partial_l^{0}\partial^m h||_{L^2_{x,v}}^2 \right)
 + K_1^{(s)}\, C_{\mathcal F}\, ||h||_{H_{x,v}^{s,r}}^2\, ||h||_{H_{\Lambda}^{s,r}}^2\,. \end{equation}
 Here all constants are independent of $z$.
 \end{proposition}
\begin{remark}
In Proposition \ref{prop}, there is a negative constant order $-1/\epsilon^2$ as the coefficient of the microscopic part $h^{\perp}$ (the first term inside the parenthesis of right-hand-side of (\ref{pr2})), which is the same order as that
derived by Guo in \cite{Guo-NS} for the dissipation rate. 
\begin{theorem}
\label{thm2}
For all $s\geq s_0$, $\exists\, (b_{j,l}^{(s)}), \, (\alpha_l^{(s)}), \, (a_{i,l}^{(s)})>0$ and $0\leq \epsilon_d\leq 1$, such that
for all $0\leq \epsilon\leq \epsilon_d$\,, \\
(1)\, $||\cdot||_{\mathcal H_{\epsilon_{\perp}}^s}\sim ||\cdot||_{H_{x,v}^s} ;$\\
(2)\, Assume  $||h_{in}||_{H_{x,v}^{s,r}L_z^{\infty}} \leq C_{I}$, then if $h$ is a solution of (\ref{INS-scaling}) in
$H_{x,v}^{s,r}$ for all $z$, we have
\begin{equation}\label{thm2_1} ||h||_{H_{x,v}^{s,r} L_z^{\infty}}\leq C_{I}\, e^{-\tau_s t}\,, \end{equation}
where $C_I$, $\tau_s$ are positive constants independent of $\epsilon$. Furthermore,
\begin{equation}\label{thm2_2} ||h||_{H_{x,v}^{s} H_z^r} \leq  C_{I}\, e^{-\tau_s t}\,. \end{equation}
\end{theorem}
\end{remark}
\begin{remark}
(i) For the acoustic scaling (\ref{Euler-scaling}), one can get similar
results as in Theorem \ref{thm1} and Proposition \ref{prop}. 
One only needs to multiply by $\epsilon$ to the right-hand-side of the estimates (\ref{Th1_3}), (\ref{Th1_4}),
(\ref{pr1}) and (\ref{pr2}). The corresponding results for Theorem \ref{thm1} become 
$$||h||_{\mathcal H_{\epsilon}^{s,r}L_z^{\infty}} \leq \delta_s\, e^{-\epsilon\tau_s t}\,, \qquad
 ||h||_{\mathcal H_{\epsilon}^s H_z^r} \leq \delta_s\, e^{-\epsilon\tau_s t}\,.$$
(2) of Theorem \ref{thm2} accordingly changes to
$$||h||_{H_{x,v}^{s,r} L_z^{\infty}}\leq C_I\, e^{-\epsilon\tau_s t}\,, \qquad
||h||_{H_{x,v}^{s} H_z^r} \leq C_I\,  e^{-\epsilon\tau_s t}\,.$$
(ii) Theorem \ref{thm2} shows that the uncertainties from the initial datum will eventually diminish and the solution will
exponentially decay in time to the deterministic global equilibrium, 
with a decay rate of $\mathcal O(e^{-t})$ under the incompressible Navier-Stokes scaling and
$\mathcal O(e^{-\epsilon t})$ under the acoustic scaling.
\end{remark}
\section{Proof of Proposition \ref{prop} and Theorem \ref{thm2}}
\label{Proof_Thms}
The proof follows the framework in \cite{MB} for deterministic equations under the incompressible Navier-Stokes scaling, since our analysis in the random space
depends on $z$ pointwisely. The main difference lies in the following:
1) one needs to check that all constants are independent of $z$, which is the case here by going through the proofs in
\cite{MB}; 2) taking $\partial^m$ of $\mathcal F$ will have crossing terms like $\mathcal F(\partial^i h, \partial^{m-i} h)$ ($0\leq i\leq m$), thus one needs to
verify Assumption {\bf H5}, which is done in section \ref{Proof_H4} for the Boltzmann equation with uncertainties.
\\[2pt]

\noindent {\it\underline{Proof of Proposition \ref{prop}:}}\,
For all $z$, one can observe (\ref{pr1}) by taking $\partial^m$ on both sides of all the estimates derived in \cite{MB}
for deterministic problems.
Summing up $m=0, \cdots, r$ in (\ref{pr1}) and using Assumption ${\bf H5}$, we get (\ref{pr2}).
\\[2pt]

\noindent {\it\underline{Proof of Theorem \ref{thm2} (2):}}\,
The proof of (1) for each $z$ is the same as in \cite{MB}.
To prove (2), for all $z$, one can easily observe the following Lemma by taking $\partial^m$ on both sides of equation (\ref{INS-scaling}) and then follow
all the estimates derived in \cite{MB}
for deterministic problems.
\begin{lemma}
\begin{align}
&\displaystyle \frac{d}{dt}||\partial^{m} h||_{\mathcal H_{\epsilon_{\perp}}^s}^2 \leq
- K_0^{(s)}\left(\sum_{|j|+|l|\leq s, |j|\geq 1}\, ||\partial_l^j \partial^m h^{\perp}||_{\Lambda}^2 +
\sum_{0\leq |l|\leq s}\, ||\partial_l^0 \partial^m h||_{\Lambda}^2\right) + K_1^{(s)}\left(\mathcal G_{x,v,z}^{s,m}(h,h)\right)^2 \notag\\[2pt]
&\displaystyle \label{pr3}\qquad\qquad\qquad\leq - K_0^{(s^{\ast})}\, ||\partial^m h||_{H_{\Lambda}^s}^2 + K_1^{(s)}\left(\mathcal G_{x,v,z}^{s,m}(h,h)\right)^2\,.
\end{align}
\end{lemma}

Then we sum up $m=0, \cdots, r$ of (\ref{pr3}) and apply Assumption ${\bf H5}$, $\exists\, s_0\in\mathbb N$, $\forall\, s\geq s_0$, such that
\begin{equation}\label{h_perp}\frac{d}{dt}||h||_{\mathcal H_{\epsilon_{\perp}}^{s,r}}^2 \leq \left( K_1^{(s)} C_{\mathcal F}\, ||h||_{H_{x,v}^{s,r}}^2 - K_0^{(s^{\ast})}\right) ||h||_{H_{\Lambda}^{s,r}}^2\,.
\end{equation}
Since $||h||_{\mathcal H_{\epsilon_{\perp}}^{s,r}}$ and $||h||_{H_{x,v}^{s,r}}$ are equivalent, so
$||h||_{H_{x,v}^{s,r}}^2 \leq C\, ||h||_{\mathcal H_{\epsilon_{\perp}}^{s,r}}^2$ with $C$ independent of $\epsilon$, then
$$\frac{d}{dt} ||h||_{\mathcal H_{\epsilon_{\perp}}^{s,r}}^2\leq \left( K_1^{(s)}C_{\mathcal F}\, C\,  ||h||_{\mathcal H_{\epsilon_{\perp}}^{s,r}}^2 - K_0^{(s^{\ast})}\right) ||h||_{H_{\Lambda}^{s,r}}^2\,. $$
Therefore if the initial data satisfy
$$||h_{in}||_{\mathcal H_{\epsilon_{\perp}}^{s,r}}^2 \leq \frac{K_0^{(s^{\ast})}}{2K_1^{(s)}C_{\mathcal F}\, C}\,, $$
one implies that $||h||_{\mathcal H_{\epsilon_{\perp}}^{s,r}}^2$ is always decreasing, so for all $t>0$,
$$ \frac{d}{dt}||h||_{\mathcal H_{\epsilon_{\perp}}^{s,r}}^2 \leq - \frac{K_0^{(s^{\ast})}}{2K_1^{(s)}C_{\mathcal F}\, C}\, ||h||_{H_{\Lambda}^{s,r}}^2\leq
- C^{\ast} ||h||_{\mathcal H_{\epsilon_{\perp}}^{s,r}}^2\,,$$
where $C^{\ast}$ is a constant independent of $z$. The last inequality is because $H_{\Lambda}^s$ controls the $H_{x,v}^s$--norm that is equivalent to the
$\mathcal H_{\epsilon_{\perp}}^s$--norm. Applying Gronwall's inequality gives
the exponential decay of $||h||_{\mathcal H_{\epsilon_{\perp}}^{s,r}}\sim ||h||_{H_{x,v}^{s,r}}$\,, thereafter the exponential decay of
$||h||_{H_{x,v}^{s,r}L_z^{\infty}}$, so (\ref{thm2_1}) is proved.
Furthermore, one has \begin{equation}\label{Hz}||h||_{H_{x,v}^{s}H_z^r}^2 = \int_{I_z}\, ||h||_{H_{x,v}^{s,r}}^2\, \pi(z)dz \leq  ||h||_{H_{x,v}^{s,r}L_z^{\infty}}^2\, \int_{I_z} \pi(z)dz \leq
C_{I}^2\, e^{-2\tau_s t}\,, \end{equation}
thus (\ref{thm2_2}) is obtained.

\section{The Boltzmann Equation with Random Inputs}
\subsection{The Basic Setup}
\label{Section2}
Let us consider the Boltzmann equation with uncertain initial data and both scalings.
For discussion of the case with random collision kernels, refer to subsection \ref{kernel}. 
The problem reads
\begin{align}
\label{BP}
\left\{
\begin{array}{l}
\displaystyle \partial_t f + \frac{1}{\epsilon^\alpha} v\cdot\nabla_x f =\frac{1}{\epsilon^{1+\alpha}}\mathcal Q(f,f), \\[4pt]
\displaystyle  f(0,x,v,z)=f^0(x,v,z), \qquad x\in\Omega\subset\mathbb T^d,  \, v\in \mathbb R^d,  \, z\in I_z,
\end{array}\right.
\end{align}
The collision operator (local in $t$, $x$) is
$$\mathcal Q(f,f) = \int_{\mathbb R^d\times{\mathbb S^{d-1}}}\, B(|v-v_{\ast}|,  \cos\theta)\, (f^{\prime}f_{\ast}^{\prime}- f f_{\ast})\,
dv_{\ast}\, d{\sigma}. $$
We adopt notations $f^{\prime}=f(v^{\prime})$, $f_{\ast}=f(v_{\ast})$ and $f_{\ast}^{\prime}=f(v_{\ast}^{\prime})$, where
$$v^{\prime}= (v+v_{\ast})/2 + (|v-v_{\ast}|/2) \sigma, \qquad v_{\ast}^{\prime} = (v+v_{\ast})/2 - (|v-v_{\ast}|/2)\sigma$$
are the post-collisional velocities of particles with pre-collisional velocities $v$ and $v_{\ast}$.
$\theta\in[0,\pi]$ is the deviation angle between $v^{\prime}-v_{\ast}^{\prime}$ and $v-v_{\ast}$.

Boltzmann's collision operator conserves mass, momentum and energy. The solution formally satisfies the celebrated Boltzmann's H theorem,
\begin{equation}\label{H}-\frac{d}{dt}\int_{\mathbb R^d}\, f \log f \, dv= -\int_{\mathbb R^d}\, \mathcal Q(f,f) \log(f)\, dv\geq 0. \end{equation}
The global equilibrium distribution is given by the Maxwellian distribution
\begin{equation}\label{Max}\mathcal M(\rho_{\infty}, u_{\infty}, T_{\infty}) = \frac{\rho_{\infty}}{(2\pi T_{\infty})^{N/2}}\, \exp\left(-\frac{|u_{\infty}-v|^2}{2T_{\infty}}\right), \end{equation}
where $\rho_{\infty}$, $u_{\infty}$, $T_{\infty}$ are the density, mean velocity and temperature of the gas
\begin{align*}
&\displaystyle \rho_{\infty} = \int_{\Omega\times{\mathbb R^d}}\, f(v)\, dxdv, \qquad
u_{\infty} = \frac{1}{\rho_{\infty}}\, \int_{\Omega\times{\mathbb R^d}}\, vf(v)\, dxdv, \\[2pt]
&\displaystyle T_{\infty} = \frac{1}{N\rho_{\infty}}\int_{\Omega\times{\mathbb R^d}}\, |u_{\infty}-v|^2\, f(v)\, dxdv,
\end{align*}
which are all determined by the initial datum due to the conservation properties.
We will consider hard potentials with $B$ satisfying Grad's angular cutoff, that is,
\\[1pt]

{\bf Assumptions on the collision kernel:}
\begin{align}
&\displaystyle B(|v-v_{\ast}|, \cos\theta) = \phi(|v-v_{\ast}|)\, b(\cos\theta),  \qquad
\phi(\xi) = C_{\phi}\, \xi^{\gamma}, \, \text{  with      }\gamma\in[0,1], \notag\\[2pt]
&\label{BK}\displaystyle \forall\eta\in[-1,1], \qquad |b(\eta)|\leq C_b, \quad |b^{\prime}(\eta)|\leq C_b\,,
\end{align}
where $b$ is non-negative and not identically equal to $0$. Introduce the collision frequency
$$\nu(v) = \int_{\mathbb R^d\times S^{d-1}}\phi(|v-v_{\ast}|) b(\cos\theta) \mathcal M(v_{\ast})\, dv_{\ast}d\sigma
=(\phi \ast\mathcal M)(v). $$
Recall that $h$ solves (\ref{INS-scaling}),
with the linearized collision operator given by
\begin{align}
&\displaystyle\mathcal L(h)= M^{-1}\left[\mathcal Q(Mh,\mathcal M) + \mathcal Q(\mathcal M, Mh)\right] = K(h)-\Lambda(h), \notag\\[4pt]
&\label{L0}\displaystyle \qquad = M \int_{\mathbb R^d\times S^{d-1}}\phi(|v-v_{\ast}|)b(\cos\theta)\mathcal M(v_{\ast}) 
\left[\frac{h_{\ast}^{\prime}}{M_{\ast}^{\prime}}+\frac{h^{\prime}}{M^{\prime}}-\frac{h_{\ast}}{M_{\ast}}-\frac{h}{M}\right] dv_{\ast}d\sigma\,. 
\end{align}
where
\begin{align*}
&\displaystyle\Lambda(h)=\nu(v)h, \qquad K(h)=\mathcal L^{+}(h)-\mathcal L^{\ast}(h), \qquad
\mathcal L^{\ast}(h)=M [(hM)\ast\phi], \\[2pt]
&\displaystyle \mathcal L^{+}(h)=\int_{\mathbb R^d\times\mathbb S^{d-1}}\phi(|v-v_{\ast}|) b(\cos\theta)\, [h^{\prime}M_{\ast}^{\prime}+
h_{\ast}^{\prime}M^{\prime}] M_{\ast}\, dv_{\ast}d\sigma\,.
\end{align*}
The bilinear part is given by
\begin{align}
&\displaystyle\mathcal F(h,h)=M^{-1}\left[\mathcal Q(Mh,Mh)+\mathcal Q(Mh,Mh)\right] \notag\\[2pt]
&\displaystyle\label{GB}\qquad\quad =\int_{\mathbb R^d\times\mathbb S^{d-1}}\phi(|v-v_{\ast}|)b(\cos\theta)
M_{\ast}(h_{\ast}^{\prime}h^{\prime}-h_{\ast}h)\, dv_{\ast}d\sigma\,.
\end{align}
The spectrum of $\mathcal L$ in $L_v^2$ is included in $\mathbb R_{-}$. Moreover the null space of $\mathcal L$ is
\begin{equation}\label{NB} N(\mathcal L)=\text{Span}\{M, \,v_1 M, \,\cdots, v_d M, \, |v|^2 M\}:=\text{Span}\{\varphi_1, \cdots, \varphi_n\}, \qquad
n=d+2\,,\end{equation}
thus \begin{equation}\label{LN} \int_{\mathbb R^d}\, \mathcal L(h) \varphi_i\, dv=0\,, \qquad i=1, \cdots, n\,.  \end{equation}
Define the coercivity norm $$||h||_{\Lambda}=||h(1+|v|)^{\gamma/2}||_{L^2}\,.$$
The coercivity argument of $\mathcal L$ is proved in \cite{MC}:
\begin{equation}
\label{CoerB} \langle h, \, \mathcal L(h)\rangle_{L^2_v}\leq -\lambda\, ||h^{\perp}||_{\Lambda_v^2}\,.
\end{equation}
Explicit spectral gap estimates for the linearized Boltzmann and Landau operators with hard potentials
have been obtained in \cite{BCM} and extended to estimates given in \cite{MC}.
Thus $\mathcal L_{B}$ satisfies Assumption $\bf{H3}$ with an explicit bound.
Proofs of Assumptions ${\bf H1^{\prime}}$, ${\bf H2^{\prime}}$, ${\bf H5}$ are given
in \cite{CN} and \cite{MB}. We will show in the following subsection that with the restriction on the collision kernel (\ref{BK}),
Assumption ${\bf H5}$ is satisfied.

\subsection{Proof of Assumption H5}
\label{Proof_H4}
The bilinear part $\mathcal F$ is given in (\ref{GB}). Let $\mathcal F(h,h)=\mathcal F^{+}(h,h) + \mathcal F^{-}(h,h)$, with
\begin{align*}
&\displaystyle \mathcal F^{+}(h,h)= \int_{\mathbb R^{d}\times S^{d-1}}\, \phi(|v-v_{\ast}|)\, b(\cos\theta)
M_{\ast}\, h_{\ast}^{\prime}\, h^{\prime}\, dv_{\ast}d\sigma\,,  \\[2pt]
&\displaystyle \mathcal F^{-}(h,h) = -\int_{\mathbb R^{d}\times S^{d-1}}\,  \phi(|v-v_{\ast}|)\, b(\cos\theta)
M_{\ast}\, h_{\ast}\, h\, dv_{\ast}d\sigma\,.
\end{align*}
Denote $\sum_{|j_0|+|j_1|+|j_2|=j,\, |l_1|+|l_2|=l} =\sum_{j,\, l}$\,. 
Differentiating operator $\mathcal F^{-}$, one obtains (see \cite{MB})
$$\partial_l^{j}\mathcal F^{-}(h,h)=-\frac{1}{2}\sum_{j,\, l}\, \int_{\mathbb R^{d}\times S^{d-1}}\,
b(\cos\theta)\, |u|^{\gamma}\, \partial_{0}^{j_0}\left(\mathcal M(v-u)^{1/2}\right)(\partial_{l_1}^{j_1}h_{\ast})(\partial_{l_2}^{j_2}h)\, dud\sigma\,.$$
Denote $\sum_{i=0}^{m}=\sum_{i}$ and $\sum_{m=0}^{r}=\sum_{m}$.
For $|m|\leq r$, take $\partial^m$,
$$ \partial^m \partial_l^{j}\mathcal F^{-}(h,h) = -\frac{1}{2}\sum_{j,\, l}\sum_{i}\binom{m}{i}\,
\int_{\mathbb R^{d}\times S^{d-1}}\, b(\cos\theta)\, |u|^{\gamma}\, \partial_{0}^{j_0}\left(\mathcal M(v-u)^{1/2}\right)
(\partial^{i}\partial_{l_2}^{j_2}h)(\partial^{m-i}\partial_{l_1}^{j_1}h_{\ast})\, dud\sigma\,. $$
Following \cite{MB} and the cutoff assumption $|b|\leq C_b$, by the Cauchy-Schwarz inequality, one has
\begin{align}
&\displaystyle |\langle\partial^{m}\partial_i^{j}\mathcal F^{-}, \, f\rangle| \leq C \sum_{j,\, l} \sum_{i}\binom{m}{i}
\int_{\Omega\times\mathbb R^{d}} (1+|v|)^{\gamma}\, |\partial^{i}\partial_{l_2}^{j_2}h|\, |f| \left(\int_{\mathbb R^{d}}\mathcal
M_{\ast}^{1/8}\, |\partial^{m-i}\partial_{l_1}^{j_1}h_{\ast}|\, dv_{\ast}\right) dvdx \notag\\[4pt]
&\label{G0}\displaystyle \qquad\qquad\qquad\leq C\, \mathcal G^{s,m}_{x,v,z}(h,h)\, ||f||_{\Lambda}\,,
\end{align}
where
\begin{align}
&\displaystyle\mathcal G^{s,m}_{x,v,z}(h,h) = \sum_{i}\binom{m}{i} \sum_{|j_1|+|l_1|+|j_2|+|l_2|\leq s}\left(\int_{\mathbb T^{d}}\, ||\partial^{i}\partial_{l_2}^{j_2}h||_{\Lambda_v}^2\, ||\partial^{m-i}\partial_{l_1}^{j_1}h||_{L^2_v}^2\, dx\right)^{1/2} \notag\\[4pt]
&\displaystyle\qquad\qquad\quad\leq \sum_{i}\binom{m}{i}\,
C_s\, ||\partial^{m-i}h||_{H_{x,v}^s}\, ||\partial^i h||_{H_{\Lambda}^s}
\notag \\[4pt]
&\label{G1}\displaystyle\qquad\qquad\quad\leq C_{s,r}\, ||h||_{H^{s,m}}\, ||h||_{H_{\Lambda}^{s,m}}
\leq C_{s,r}\, ||h||_{H^{s,r}}\, ||h||_{H_{\Lambda}^{s,r}}\,,
\end{align}
where the constant $C_{s,r}$ depends on $s$ and $r$, and H$\ddot{o}$lder's inequality was used in the second inequality.
To get the first inequality, note that the Sobolev embedding stating that $\exists\, s_0\in \mathbb N$,
such that if $s\geq s_0$, we have $H_{x}^{s/2} \hookrightarrow L_{x}^{\infty}$. We divide the sum into two cases $|j_1|+|l_1|\leq s/2$ and $|j_2|+|l_2|\leq s/2$\,.
If $|j_1|+|l_1|\leq s/2$, then for each $z$, one has
\begin{align*}
&\displaystyle\qquad ||\partial_{l_1}^{j_1}\partial^{m-i}h||_{L^2_{v}}^2 \leq \sup_{x\in\mathbb T^{d}}\,  ||\partial_{l_1}^{j_1}\partial^{m-i}h||_{L^2_{v}}^2
\leq \widetilde C_{s} \left|\left|\, ||\partial_{l_1}^{j_1}\partial^{m-i}h||_{L^2_v}^2\, \right|\right|_{H_{x}^{s/2}} \\[2pt]
&\displaystyle\quad = C_s \sum_{|p|\leq s/2}\sum_{p_1+p_2=p}\, \int_{\mathbb T^{d}\times\mathbb R^{d}} \left| (\partial_{l_1+p_1}^{j_1}\partial^{m-i}h)\, (\partial_{l_1+p_2}^{j_1}
\partial^{m-i}h)\right| dvdx \leq C_s\, ||\partial^{m-i}h||_{H_{x,v}^{s}}^2\,,
\end{align*}
after using the Cauchy-Schwarz inequality in the last step. 
In the other case, if $|j_2|+|l_2|\leq s/2$ and by the same calculations,
$ ||\partial_{l_2}^{j_2}\partial^i h||_{\Lambda_v}^2 \leq C_s\, ||\partial^i h||_{H_{\Lambda}^{s}}^2$\,. Taking the square and summing up
$|m|=0, \cdots, r$ on both sides of (\ref{G1}), we obtain Assumption ${\bf H5}$. \\
The second term $\mathcal F_{B}^{+}$ is dealt with in the same way.

\begin{remark}
Assumptions ${\bf H1-H5}$, ${\bf H1^{\prime}-H2^{\prime}}$ introduced in section \ref{sec1} for the hypocoercivity theory also
hold for several different kinetic models, in addition to the Boltzmann equation.
\cite{CN} validates the assumptions for linear relaxation, linear Fokker-Planck,
nonlinear semi-classical quantum relaxation kinetic and the Landau equation with hard and moderately soft potential.
The results established in section \ref{Thms} for Boltzmann equation can be done for these other models in a similar fashion.
We omit the details. 
\end{remark}

\subsection{Uncertainties from random collision kernels}
\label{kernel}
The above estimate can also be applied to the case of random collision kernels, in which $b$ depends on $z$,
under the restriction that all the $z$-derivatives of $b$ are assumed to be bounded, i.e.,
\begin{align}
&\displaystyle B(|v-v_{\ast}|, \cos\theta, z) = \phi(|v-v_{\ast}|)\, b(\cos\theta,z),  \qquad
\phi(\xi) = C_{\phi}\, \xi^{\gamma}, \,\text{with      }\gamma\in[0,1], \notag\\[4pt]
&\label{BK1}\displaystyle \forall\eta\in[-1,1], \qquad |b(\eta,z)|\leq C_b, \, |\partial_{\eta} b(\eta,z)|\leq C_b, \,  \text{  and }\,
|\partial_z^k b(\eta, z)|\leq C_b^{\ast},  \,  \forall\, 0\leq k\leq r\,.
\end{align}

Under the above assumptions for $B$, 
all the Assumptions H1-H5, H1'-H2' introduced in section \ref{sec1} still hold when uncertainties are from collision kernels, 
which is obvious to see. The main goal of this section is to show that 
Theorem \ref{thm2} is still valid when collision kernels are random. 
The difference in the proof compared to the random initial data case will be discussed. 
The impact of $z$ dependence of $b$ on the linear (or linearized) collision operator $\mathcal L$ needs to be
estimated differently, which is shown in Lemma \ref{LL1}. 
With Lemma \ref{LL1}, we then show how to embed it and get the same result as Theorem \ref{thm2}
--with rest of the proof following the case of random initial data. 
This section will be concluded with Theorem \ref{random_kernel}. 

Below we use the linearized Boltzmann operator with random collision kernel (\ref{BK1}) as an example, although the same analysis can also be done 
for the linearized Landau, semi-classical relaxation model of quantum Boltzmann equation and the linear Fokker-Planck equation. 
Our analysis follows those of \cite{Ma, Liu} for linear transport equations, with some new estimates by using the $\Lambda$-norm.

Suppose $g$ is a solution to 
\begin{equation}\label{g1}\partial_t g +\frac{1}{\epsilon}v\cdot\nabla_x g= \frac{1}{\epsilon^2}\mathcal L(g), \end{equation}
where $\mathcal L$ is given in (\ref{L0}), with $b$ given by (\ref{BK1}). \\
Recall (\ref{Pi}) that $\displaystyle\Pi_{\mathcal L} (g) = \sum_{i=1}^{n} \left(\int_{\mathbb R^d}\, g\varphi_i\, dv \right)\varphi_i$. Then by (\ref{LN}), one knows
\begin{align*}
&\displaystyle\int_{\Omega\times\mathbb R^d}\, \mathcal L(h)\, \Pi_{\mathcal L}(g)\, dvdx
=\int_{\Omega\times\mathbb R^d}\,  \mathcal L(h)\, \sum_{i=1}^{n}\left(\int_{\mathbb R^d} g\varphi_i\, dv\right)\varphi_i\, dvdx \\[2pt]
&\displaystyle \qquad\qquad\qquad\qquad\qquad=\int_{\Omega}\, \sum_{i=1}^{n}\, \underbrace{\int_{\mathbb R^d}\, \mathcal L(h)\varphi_i\, dv}_{=0}
\left(\int_{\mathbb R^d}\, g\varphi_i\,dv\right) dx=0\,.
\end{align*}
Since $\mathcal L(h)=\mathcal L(h-\Pi_{\mathcal L}(h))=\mathcal L(h^{\perp})$ because
$\Pi_{\mathcal L}(h) \in N(\mathcal L)$, thus 
\begin{equation}\label{LLL}\langle \mathcal L(h), \, g\rangle_{L^2_{x,v}}=\langle \mathcal L(h), \, g-\Pi_{\mathcal L}(g)\rangle_{L^2_{x,v}}=
\langle \mathcal L(h^{\perp}), \, g^{\perp}\rangle_{L^2_{x,v}}\leq C^{\mathcal L} ||h^{\perp}||_{\Lambda}\, ||g^{\perp}||_{\Lambda}\,, 
\end{equation}
where we used (\ref{LL}) in Assumption ${\bf H1}$. 

The below arguments follow the idea in \cite{Ma, Liu}, nevertheless, with the help of (\ref{LLL}), 
a more general framework (especially for the linear Fokker-Planck and 
the linearized Boltzmann, Landau, etc) is constructed here. 
Denote $$||g||_{L_{x,v}^{2,r}}^2 = \sum_{|m|\leq r}\, ||\partial^m g||_{L^2_{x,v}}^2, \qquad
||g||_{\Lambda^{r}}^2 = \sum_{|m|\leq r}\, ||\partial^{m}g||_{\Lambda}^2\,. $$
Our goal is to show that 
\begin{equation}\label{Cm}\epsilon^2 \partial_t \bigg(\sum_{m=0}^{r} \widetilde C_{m,r+1}\,  ||\partial^m g||_{L^2_{x,v}}^2\bigg) \leq -\lambda\, ||g^{\perp}||
_{\Lambda^{r}}^2\,,\end{equation}
where $\lambda>0$ is a $z$-independent constant that depends on $C_b$ in (\ref{BK1}), $\widetilde C_{m,r+1}>0$ are constants. 

Taking $\partial^l$ on (\ref{g1}), one gets
\begin{equation}\label{g2}\epsilon^2 \partial_t \partial^{l}g + \epsilon v\cdot\nabla_x (\partial^{l}g) = \partial^{l}\mathcal L(g)\,.\end{equation}
Multiplying (\ref{g2}) by $\partial^l g$ and integrating on $x$ and $v$, then
\begin{equation}\label{g3}\frac{\epsilon^2}{2} \partial_t ||\partial^l g||_{L^2_{x,v}}^2 = \langle \partial^{l}\mathcal L(g), \, \partial^{l}g\rangle_{L^2_{x,v}}, 
\end{equation}
where integration by parts and periodic boundary condition in $x$ are used, thus $\langle v\cdot\nabla_x(\partial^{l}g), \, \partial^l g\rangle_{L^2_{x,v}}=0$. 
We first prove the following Lemma. 
\begin{lemma}
\label{LL1}
For any $m\geq 0$, there exist $m$ constants $C_{jm}>0$, $j=0, \cdots, m-1$ such that 
\begin{align}
\label{MI}\displaystyle \epsilon^2 \partial_t \bigg(||\partial^m g||_{L^2_{x,v}}^2 + \sum_{j=0}^{m-1}C_{jm}\, ||\partial^j g||_{L^2_{x,v}}^2\bigg) \leq 
\begin{cases}   
  -2\lambda\, ||g^{\perp}||_{\Lambda}^2\,,\qquad m=0, \\[4pt]
  - \lambda\, ||\partial^m g^{\perp}||_{\Lambda}^2 \,,\quad m\geq 1. 
\end{cases}
 \end{align}
 \end{lemma}
 \begin{proof}
We will prove this lemma by using Mathematical Induction. \\
When $m=0$, by (\ref{g3}), (\ref{MI}) holds because of Assumption ${\bf H3}$. 
Assume that (\ref{MI}) holds for any $m\leq p$, where $p\in \mathbb N$. Adding all these inequalities, we get
$$ \epsilon^2 \partial_t \bigg(\frac{1}{2}||g||_{L^2_{x,v}}^2 + \sum_{m=1}^{p} ||\partial^m g||_{L^2_{x,v}}^2 + 
\sum_{m=1}^{p} \sum_{j=0}^{m-1} C_{jm}\, ||\partial^j g||_{L^2_{x,v}}^2 \bigg)\leq -\lambda\,||g^{\perp}||_{\Lambda^{p}}\,,$$
which is equivalent to 
\begin{equation}\label{p0}\epsilon^2 
\partial_t \bigg(\sum_{j=0}^p C_{j,p+1}^{\prime}\, ||\partial^j g||_{L^2_{x,v}}^2 \bigg) \leq -\lambda\, ||g^{\perp}||_{\Lambda^{p}}^2\,,
\end{equation}
where 
\begin{align*}
\displaystyle C_{j,p+1}^{\prime} =
 \begin{cases} \frac{1}{2} + \sum_{m=1}^{p} C_{0m}, \qquad j=0, \\[2pt]
                        1 + \sum_{m=1}^{p} C_{jm}, \qquad 1\leq j\leq p-1, \\[2pt]       
                        1, \qquad\qquad\qquad\qquad j=p.           
 \end{cases}
\end{align*}
Denote
\begin{equation}\label{Li}\mathcal L_{i}(g) = M \int_{\mathbb R^d\times S^{d-1}} \phi(|v-v_{\ast}|)\partial^{p+1-i}b\, \mathcal M(v_{\ast}) 
\left[\frac{\partial^i g_{\ast}^{\prime}}{M_{\ast}^{\prime}}+\frac{\partial^i g^{\prime}}{M^{\prime}}-\frac{\partial^i g_{\ast}}{M_{\ast}}-\frac{\partial^i g}{M}\right] dv_{\ast}d\sigma\,,\end{equation}
where $|\partial^{p+1-i}b|\leq C_b^{\ast}$ as shown in (\ref{BK1}).

When $m=p+1$, let $l=p+1$ in (\ref{g3}), and the right-hand-side of (\ref{g3}) has the following estimate: 
\begin{align}
&\displaystyle \quad\langle \partial^{p+1}\mathcal L(g), \, \partial^{p+1}g\rangle_{L^2_{x,v}} = \langle\mathcal L(\partial^{p+1} g), \, \partial^{p+1} g\rangle_{L^2_{x,v}}
+ \sum_{i=0}^{p}\binom{p+1}{i} \langle\mathcal L_{i}(g), \, \partial^{p+1} g\rangle_{L^2_{x,v}}\notag\\[4pt]
&\displaystyle  \leq -\lambda\, ||\partial^{p+1} g^{\perp}||_{\Lambda}^2 + C^{\mathcal L, \ast}\, 
||\sum_{i=0}^{p} \binom{p+1}{i}\, \partial^{i}g^{\perp}||_{\Lambda}\, ||\partial^{p+1} g^{\perp}||_{\Lambda}\notag\\[4pt]
&\displaystyle \leq  -\lambda\, ||\partial^{p+1} g^{\perp}||_{\Lambda}^2 + \frac{(C^{\mathcal L, \ast})^2}{2\lambda}\, 
||\sum_{i=0}^{p} \binom{p+1}{i}\, \partial^{i}g^{\perp}||_{\Lambda}^2+\frac{\lambda}{2}\, ||\partial^{p+1} g^{\perp}||_{\Lambda}^2 \notag\\[4pt]
&\label{p+1}\displaystyle \leq -\frac{\lambda}{2}\, ||\partial^{p+1} g^{\perp}||_{\Lambda}^2 + 
4^{p+1}\,\frac{(C^{\mathcal L, \ast})^2}{2\lambda}\, \sum_{i=0}^{p}\, ||\partial^{i}g^{\perp}||_{\Lambda}^2 
= -\frac{\lambda}{2}\, ||\partial^{p+1} g^{\perp}||_{\Lambda}^2+
4^{p+1}\,\frac{(C^{\mathcal L, \ast})^2}{2\lambda}\, ||g^{\perp}||_{\Lambda^{p}}\,,
\end{align}
where we used (\ref{LLL}) with $C^{\mathcal L}$ substituted by $C^{\mathcal L, \ast}$ that depends on $C_b^{\ast}$, since $\mathcal L_{i}$, as defined by 
(\ref{Li}), has $z$-derivatives of $b$ involved. Young's inequality is used in the second inequality, and 
the Cauchy-Schwartz inequality, namely 
$$||\sum_{i=0}^{p}\binom{p+1}{i}\, \partial^{i}g^{\perp}||_{\Lambda}^2 \leq \sum_{i=0}^{p}\binom{p+1}{i}^2\, \sum_{i=0}^{p} ||\partial^{i}g^{\perp}||_{\Lambda}^2 
\leq 4^{p+1}\, \sum_{i=0}^{p} ||\partial^{i}g^{\perp}||_{\Lambda}^2 $$
is used in the third inequality.
Therefore, 
\begin{equation}\label{pp}\epsilon^2 \partial_t ||\partial^{p+1}g||_{L^2_{x,v}}^2 \leq -\lambda\, ||\partial^{p+1}g^{\perp}||_{\Lambda}^2 + 4^{p+1}\frac{(C^{\mathcal L, \ast})^2}{\lambda}\, ||g^{\perp}||_{\Lambda^{p}}\,.\end{equation}

Multiplying (\ref{p0}) by 
$\displaystyle\frac{4^{p+1}(C^{\mathcal L, \ast})^2}{\lambda^2}$ and adding to (\ref{pp}), one has
$$\epsilon^2\partial_t \bigg( ||\partial^{p+1}g||_{L^2_{x,v}}^2 + \sum_{j=0}^{p}C_{j,p+1}\, ||\partial^j g||_{L^2_{x,v}}^2\bigg) 
\leq -\lambda\, ||\partial^{p+1}g^{\perp}||_{\Lambda}^2\,,$$
where 
$\displaystyle C_{j,p+1}=\frac{4^{p+1}(C^{\mathcal L, \ast})^2}{\lambda^2}\, C_{j,p+1}^{\prime}$. 
This shows that (\ref{MI}) holds when $m=p+1$. By Mathematical Induction, (\ref{MI}) is true for all $m\in \mathbb N$, thus 
Lemma \ref{LL1} is proved. 
\end{proof}

As a consequence, adding up (\ref{MI}) for $m\leq r$, with $C_{m,r+1}^{\prime}=\widetilde C_{m,r+1}>0$ in (\ref{Cm}) and (\ref{p0}), 
one obtains the goal estimate (\ref{Cm}). Define 
\begin{equation}\label{weight} ||g||_{L_{x,v}^{2,r\ast}}:= \sum_{m=0}^{r} \widetilde C_{m,r+1}\, ||\partial^m g||_{L^2_{x,v}}, \end{equation}
then \begin{equation}\label{LL2}\epsilon^2 \partial_t ||g||_{L_{x,v}^{2,r\ast}} \leq -\lambda\, ||g^{\perp}||_{\Lambda^r}^2 \leq -\lambda\, c\, ||g^{\perp}||_{\Lambda^{r\ast}}^2\,, \end{equation}
where the second inequality is because (\ref{weight})--defined as some {\it weighted} Sobolev norm
$||\cdot||_{H_z^{r\ast}}$-- is equivalent to the standard Sobolev norm $||\cdot||_{H_z^{r}}$ in the random space. 

The above analysis also holds for $||\nabla_x g||_{L_{x,v}^{2,r}}$ etc, in the definition of $||g||_{\mathcal H_{\epsilon_{\perp}}^{1,r}}$. For example, 
$$\epsilon^2 \partial_t ||\nabla_x g||_{L_{x,v}^{2,r\ast}} \leq -\lambda\, c\, ||\nabla_x g^{\perp}||_{\Lambda^{r\ast}}^2\,. $$
One can extend to higher Sobolev space $\mathcal H_{\epsilon_{\perp}}^{s,r}$ and get the conclusions in Theorem \ref{thm2}. \\[4pt]

The impact of $z$ in $b(\cos\theta,z)$ on the nonlinear term can be estimated similarly to (\ref{G0}), namely one has $|\langle\partial^{m}\partial_i^{j}\mathcal F^{-}, \, f\rangle| \leq C\, \mathcal G^{s,m}_{x,v,z}(h,h)\, ||f||_{\Lambda}$, 
where 
\begin{align}
&\displaystyle \mathcal G^{s,m}_{x,v,z}(h,h) =
\left|\sum_{i=0}^{m}\binom{m}{i} \partial^{m-i}b \, \sum_{n=0}^{i}\binom{i}{n}
\sum_{|j_1|+|l_1|+|j_2|+|l_2|\leq s}\left(\int_{\mathbb T^{d}} ||\partial^{n}\partial_{l_2}^{j_2}h||_{\Lambda_v}^2\, 
||\partial^{i-n}\partial_{l_1}^{j_1}h||_{L^2_v}^2\, dx\right)^{1/2}\right| \notag\\[4pt]
&\displaystyle\qquad\qquad\leq C_b^{\ast}\, \sum_{i=0}^{m}\sum_{n=0}^{i}\binom{m}{i}\binom{i}{n}\, C_s\, ||\partial^n h||_{H_{x,v}^{s}}\, ||\partial^{i-n}h||_{H_{\Lambda}^{s}} \notag\\[4pt]
&\label{H5_b}\displaystyle\qquad\qquad\leq C_b^{\ast}\, \widetilde C_{s,r}\, ||h||_{H_{x,v}^{s,r}}\, ||h||_{H_{\Lambda}^{s,r}}
\leq C_b^{\ast}\, \widetilde C_{s,r}\, c^{\prime}\, ||h||_{H_{x,v}^{s,r\ast}}\, ||h||_{H_{\Lambda}^{s,r\ast}}\,,
\end{align}
with the constant $\widetilde C_{s,r}$ depending on $s$ and $r$, and the last inequality is because $||\cdot||_{H_z^{r}}\sim ||\cdot||_{H_z^{r\ast}}$. 
One can also assume $\phi$ depending on $z$,
and easily obtain a similar estimate upon a suitable assumption on $C_{\phi}$ in (\ref{BK1}). We omit the details.
With the estimates (\ref{LL2}) and (\ref{H5_b}) and an extension to the higher-order Sobolev space in $(x,v)$, one gets
$$||h||_{H_{x,v}^{s} H_z^{r\ast}} \leq  C_{I}\, e^{-\widetilde\tau_s t}. $$ 
Since $||\cdot||_{H_{x,v}^{s} H_z^{r\ast}}$ is equivalent to the standard Sobolev norm $||\cdot||_{H_{x,v}^{s} H_z^{r}}$ in $(x,v,z)$-space, thus
$$||h||_{H_{x,v}^{s} H_z^{r}} \leq  C_{I}\, e^{-\widetilde\tau_s t}, $$ 
where $\widetilde\tau_s>0$ is a constant independent of $z$ and $\epsilon$. 
Therefore, Theorem \ref{thm2} holds true if the uncertainty comes from random collision kernel. 

\begin{remark}
For the linear Fokker-Planck equation with small scalings, 
\begin{equation}\label{FP}\partial_t f +\frac{1}{\epsilon^{\alpha}}v\cdot\nabla_x f = \frac{\sigma(z)}{\epsilon^{1+\alpha}}\, \nabla_v\cdot (\nabla_v f + fv)\,,
\end{equation}
where $\alpha=0$ or $1$, the coercivity norm $\Lambda$ is given in \cite{MB} for deterministic problem, 
$$||h||_{\Lambda}^2 = ||vh||_{L^2_{x,v}}^2 + ||\nabla_v h||_{L^2_{x,v}}^2\,. $$
The analysis in this section can be applied to (\ref{FP}) with random diffusion coefficient $\sigma$, 
under the assumption
$|\partial_z^k \sigma|\leq C_{\sigma}$ for $k=0, \cdots, r$ with $C_{\sigma}$ a positive constant independent of $z$. 
\end{remark}

To conclude, we summarize our results in the following Theorem:

\begin{theorem}
\label{random_kernel}
If the uncertainties come from both initial data and/or collision kernel, under our assumptions
for the collision kernel (\ref{BK1}), and for the initial data 
$||h_{in}||_{H_{x,v}^{s,r} L_z^{\infty}} \leq C_{I}$, then we have the following: \\
(i) Under the incompressible Navier-Stokes scaling, 
$$||h||_{H_{x,v}^{s,r} L_z^{\infty}}\leq C_I\, e^{-\tau_s t}\,, \qquad
||h||_{H_{x,v}^{s} H_z^r} \leq C_I\,  e^{-\tau_s t}\,.$$
(ii) Under the acoustic scaling, 
$$||h||_{H_{x,v}^{s,r} L_z^{\infty}}\leq C_I\, e^{-\epsilon\tau_s t}\,, \qquad
||h||_{H_{x,v}^{s} H_z^r} \leq C_I\,  e^{-\epsilon\tau_s t}\,,$$
where $C_I$, $\tau_s$ are positive constants independent of $\epsilon$. 
\end{theorem}

\section{Spectral Accuracy of the gPC-SG Method}
\label{gPC}
\subsection{A gPC based Stochastic Galerkin Method}

In this subsection, we review the gPC-SG method for solving kinetic equations with uncertainties.
Take the Boltzmann equation as an example \cite{Hu}. One seeks for a solution in the following form:
\begin{align}
&\displaystyle  f(t,x,v,z)\approx \sum_{|\bk|=1}^{K}\, f_{\bk}(t,x,v)\psi_{\bk}(z) := f^K(t,x,v,z), \notag\\[2pt]
&\displaystyle \label{ans} h(t,x,v,z)\approx \sum_{|\bk|=1}^{K}\, h_{\bk}(t,x,v)\psi_{\bk}(z) := h^K(t,x,v,z).
\end{align}
Here $\bk=(k_1, \cdots, k_n)$ is a multi-index with $|\bk|=k_1+\cdots k_n$. $\pi(z)$ is the probability distribution function of $z$, which is given
{\it a priori} in our problem.
$\{\psi_{\bk}(z)\}$ are orthonormal gPC basis functions satisfying
$$\int_{I_z}\, \psi_{\bk}(z)\psi_{\bj}(z)\, \pi(z)dz=\delta_{\bk\bj}, \qquad 1\leq |\bk|, \, |\bj|\leq K. $$
One can expand $f$ by
$$f(t,x,v,z)=\sum_{|\bk|=1}^{\infty} \hat f_{\bk}(t,x,v)\psi_{\bk}(z), \qquad \hat f_{\bk}(t,x,v)=\int_{I_z}\, f(t,x,v,z)\psi_{\bk}(z)\,\pi(z)dz. $$
Define the projection operator $P_{K}$ as
\begin{equation}\label{proj} P_{K} f(t,x,v,z)=\sum_{|\bk|=1}^K\, \hat f_{\bk}(t,x,v)\psi_{\bk}(z). \end{equation}
Assume the random collision kernel has the assumptions given by (\ref{BK1}).
Inserting ansatz (\ref{ans}) into (\ref{INS-scaling})  and conducting a standard Galerkin projection, one obtains the gPC-SG
system for $h_{\bk}$:
\begin{align}
\label{h_gPC1}
\left\{
\begin{array}{l}
\displaystyle \partial_t h_{\bk} + \frac{1}{\epsilon}v\cdot\nabla_x h_{\bk} =\frac{1}{\epsilon^2}\mathcal L_{\bk}(h^K) +
\frac{1}{\epsilon}\mathcal F_{\bk}(h^K, h^K),
\\[2pt]
\displaystyle  h_{\bk}(0,x,v)=h_{\bk}^{0}(x,v), \qquad x\in\Omega\subset\mathbb T^d,  \, v\in \mathbb R^d,
\end{array}\right.
\end{align}
for each $1\leq |\bk|\leq K$, with a periodic boundary condition and the initial data given by
$$h_{\bk}^{0} := \int_{I_z}\, h^{0}(x,v,z)\psi_{\bk}(z)\, \pi(z)dz. $$
The collision parts are given by
\begin{align}
\begin{array}{l}
\label{LFK}
\displaystyle \mathcal L_{\bk}(h^K)=K_{\bk}(h^K)-\Lambda_{\bk}(h^K), \qquad
K_{\bk}(h^K)=\mathcal L^{+}_{\bk}(h^K) -\mathcal L^{\ast}_{\bk}(h^K), \qquad
\Lambda_{\bk}(h^K)=\sum_{|\bi|=1}^{K}\, \nu_{\bk\bi}\, h_{\bi}, \\[1pt]
\displaystyle \mathcal L^{+}_{\bk}(h^K)=\sum_{|\bi|=1}^{K}\,\int_{\mathbb R^d\times\mathbb S^{d-1}}\,\widetilde S_{\bk\bi}\,
 \phi(|v-v_{\ast}|)\, (h_{\bi}(v^{\prime}) M(v_{\ast}^{\prime}) + h_{\bi}(v_{\ast}^{\prime})M(v^{\prime}))\, M(v_{\ast})\, dv_{\ast}d\sigma, \\[1pt]
 \displaystyle \mathcal L^{\ast}_{\bk}(h^K) = M(v)\, \sum_{|\bi|=1}^{K}\,\int_{\mathbb R^d\times\mathbb S^{d-1}}\, \widetilde S_{\bk\bi}\,
 \phi(|v-v_{\ast}|)\, h_{\bi}(v_{\ast}) M(v_{\ast})\, dv_{\ast}d\sigma, \\[1pt]
\displaystyle \mathcal F_{\bk}(h^K, h^K) (t,x,v)=\sum_{|\bi|, |\bj|=1}^{K}\,\int_{\mathbb R^d\times\mathbb S^{d-1}}\, S_{\bk\bi\bj}\,
\phi(|v-v_{\ast}|)\, M(v_{\ast})\, (h_{\bi}(v^{\prime})h_{\bj}(v_{\ast}^{\prime}) - h_{\bi}(v)h_{\bj}(v_{\ast}))\, dv_{\ast}d\sigma,
\end{array}
\end{align}
with
\begin{align*}
\begin{array}{l}
\displaystyle\widetilde S_{\bk\bi}:=\int_{I_z}\, b(\cos\theta,z)\, \psi_{\bk}(z)\psi_{\bi}(z)\, \pi(z)dz, \qquad
\nu_{\bk\bi}:=\int_{\mathbb R^d\times\mathbb S^{d-1}}\, \widetilde S_{\bk\bi}\, \phi(|v-v_{\ast}|)\, \mathcal M(v_{\ast})\, dv_{\ast}d\sigma,  \\[2pt]
\displaystyle \qquad\text{and}\qquad\quad S_{\bk\bi\bj}:=\int_{I_z}\,  b(\cos\theta,z)\, \psi_{\bk}(z)\psi_{\bi}(z)\psi_{\bj}(z)\, \pi(z)dz.
\end{array}
\end{align*}

\subsection{Hypocoercivity Estimate of the gPC Solution}
\label{gPC1}
In this and the next sections, we assume $z\in I_z$ is one dimensional and $I_z$ has finite support $|z|\leq C_z$ (which is the case, for example, for the uniform and Beta distributions). Let us first introduce the main result of this section on the estimate of the gPC solution:
\begin{theorem}
\label{thm3}
Assume the collision kernel $B$ satisfies (\ref{BK1}) and is linear in $z$, with the form of
\begin{equation}\label{b1} b(\cos\theta, z)= b_0(\cos\theta)+  b_1(\cos\theta)z\,,  \end{equation}
with $|\partial_z b|\leq \mathcal O(\epsilon)$.
We also assume the technical condition
\begin{equation}\label{basis}
||\psi_k||_{L^{\infty}} \leq C k^p, \qquad \forall\, k,  \end{equation}
with a parameter $p \geq 0$. Let $q>p+2$, define the energy $E^{K}$ by
\begin{equation}\label{E^K}
  E^{K}(t) = E_{s,q}^{K}(t) = \sum_{k=1}^{K}\, ||k^q h_k||_{H_{x,v}^s}^2,
\end{equation}
with the initial data satisfying $E^{K}(0) \leq \eta$.
Then for all $s\geq s_0$, $0\leq \epsilon_d\leq 1$, such that
for $0\leq \epsilon\leq \epsilon_d$, if $h^K$ is a gPC solution of (\ref{h_gPC1}) in
$H_{x,v}^s$, we then have the following:  \\
(i) Under the incompressible Navier-Stokes scaling, 
$$E^{K}(t)\leq \eta \, e^{-\tau t}\,. $$
(ii) Under the acoustic scaling, 
$$E^{K}(t) \leq \eta \, e^{-\epsilon \tau t}\,, $$
where $\eta$, $\tau$ are all positive constants that only depend on $s$ and $q$, independent of $K$ and $z$.
\end{theorem}
\begin{remark}
The choice of energy  $E^K$ in (\ref{E^K}) enables one to obtain
the desired energy estimates with initial data independent of $K$ \cite{Rui}.
\end{remark}

To prove Theorem \ref{thm3} on estimate of the gPC solution, a modification of Assumption ${\bf H5}$ is necessary. 
\\[1pt]

{\bf Assumption H6:}
There exist constants $\delta, \, C(\delta)>0$ that are independent of $K$, such that
\begin{equation}\label{H6_1} \left|\,\sum_{k=1}^K\, k^{2q}\, \langle\partial_l^j \mathcal F_{k}(h^K, h^K), \, f_{k}\rangle_{L^2_{x,v}}\, \right|
\leq C(\delta)\, \sum_{m=1}^K\, ||m^q h_m||_{H_{\Lambda}^s}^2 \, \sum_{n=1}^K\, ||n^q h_n||_{H_{x,v}^s}^2 + \delta\,\sum_{k=1}^K\, ||k^q f_k||_{\Lambda}^2\,. \end{equation}
In order to obtain the estimate for the gPC coefficients $h_k$, we make the assumption (\ref{basis}) on the basis functions. Since $|b|\leq C_b$,
$|b_1|\leq \xi$ and $|z|\leq C_z$, then
\begin{equation}\label{S_mnk} |S_{mnk}|\leq (C_b + \xi C_z)\, ||\psi_n||_{L^{\infty}}\, \langle |\psi_m|, \, |\psi_k|\rangle_{L^2_z} \leq (C_b + \xi C_z)\, ||\psi_n||_{L^{\infty}}\, ||\psi_m||_{L^2_z}\, ||\psi_k||_{L^2_z}=\widetilde C\, n^p\,, \end{equation}
where $\widetilde C= C (C_b + \xi\, C_z)$ with $C$ given in (\ref{basis}).
We mention that this assumption was introduced in \cite{Rui}, and some examples where (\ref{S_mnk}) holds are given there.
For the case $I_z=[-1,1]$ with uniform distribution, $\psi_k$ is the normalized Legendre polynomials,
and (\ref{S_mnk}) holds with $p=1/2$. For the case $I_z=[-1,1]$ with the distribution $\pi(z)=\frac{2}{\sqrt{\pi \sqrt{1-z^2}}}$ and
$\psi_k$ are the normalized Chebyshev polynomials, (\ref{S_mnk}) holds with $p=0$.  \\[1pt]

Now since we assume that $B$ is linear in $z$ and $\psi_k$ is a $(k-1)$-th degree polynomial, orthogonal to all lower order polynomials, thus
$S_{mnk}=0$ if $(m-1)+(n-1)+1<k-1$. Then $S_{mnk}$ may be
nonzero only when the inequality
\begin{equation}\label{mnk} m+n\geq k \end{equation} holds.
Note that (\ref{S_mnk}) and (\ref{mnk}) also hold if $m, n, k$ are permuted, that is, when the inequalities
\begin{equation}\label{mnk1} m+n\leq k, \qquad \text{or  }\, n+k\leq m,\qquad \text{or }\, k+m\leq n, \end{equation}
are satisfied, $S_{mnk}$ may be nonzero.
To validate Assumption ${\bf H6}$, we follow a similar proof as Assumption ${\bf H5}$ in section \ref{Proof_H4}, combining the idea used in \cite{Rui}.
First consider $m\geq n$, by (\ref{S_mnk}) and (\ref{mnk}), then $\displaystyle \widetilde C\, m^q\, n^q \geq \bigg(\frac{k}{2}\bigg)^q\, |S_{mnk}|\, n^{q-p}$, thus
\begin{equation}\label{S1}\frac{k^{2q}}{m^q\, n^q}\, |S_{mnk}| \leq \widetilde C\, k^q\, n^{p-q}. \end{equation}
Now let $\chi_{mnk}$ be the indicator function of the set of indexes $(m,n,k)$ for which $S_{mnk} \neq 0$, namely
\begin{align}
&\label{chi}\displaystyle \chi_{mnk} = \begin{cases} 0, \qquad S_{mnk}=0, \\[2pt]
 1, \qquad S_{mnk}\neq 0,  \end{cases}
\end{align}
then
\begin{align*}
&\displaystyle \qquad \left|\,\sum_{k=1}^K\, k^{2q}\, \langle\partial_l^j \mathcal F_{k}(h^K, h^K), \, f_{k}\rangle_{L^2_{x,v}}\, \right|
  \\[2pt]
&\displaystyle \leq  \sum_{k=1}^K\, k^{2q}\sum_{m,n=1}^K \chi_{mnk}\, |S_{mnk}|\, C_s\sum_{|j_1|+|l_1|+|j_2|+|l_2|\leq s}\,
\left( \int_{\mathbb T^d}\, ||\partial_{l_2}^{j_2}h_m||_{\Lambda_v}^2\, ||\partial_{l_1}^{j_1}h_n||_{L^2}^2\, dx\right)^{1/2}\, ||f_k||_{\Lambda}\\[2pt]
&\displaystyle \leq \sum_{k,m,n=1}^K\,\frac{k^{2q}}{m^q\, n^q}\, |S_{mnk}|\, \chi_{mnk}\, C_s\, ||m^q h_m||_{H_{\Lambda}^s}\,||n^q h_n||_{H_{x,v}^s}\,
||f_k||_{\Lambda} \\[2pt]
&\displaystyle\leq\sum_{k,m,n=1}^K \, C_{s}\, \widetilde C\, n^{p-q}\, ||m^q h_m||_{H_{\Lambda}^s}\, ||n^q h_n||_{H_{x,v}^s}\,  ||k^q f_k||_{\Lambda}\,
\chi_{mnk}\\[2pt]
&\displaystyle\leq C_{s}\, \widetilde C\, C(\delta^{\prime}) \underbrace{\sum_{k,m,n=1}^K\, n^{p-q} \, ||m^q h_m||_{H_{\Lambda}^s}^2\, ||n^q h_n||_{H_{x,v}^s}^2\,\chi_{mnk}}_{I} +\,
C_{s}\, \widetilde C\, \delta^{\prime}\, \underbrace{\sum_{k,m,n=1}^K\, n^{p-q}\, ||k^q f_k||_{\Lambda}^2\,\chi_{mnk}}_{II} \\[2pt]
&\displaystyle \leq  C(\delta)\sum_{m=1}^K\, ||m^q h_m||_{H_{\Lambda}^s}^2 \, \sum_{n=1}^K\, ||n^q h_n||_{H_{x,v}^s}^2 + \delta\, \sum_{k=1}^K\, ||k^q f_k||_{\Lambda}^2\,.
\end{align*}
We used (\ref{G0}) and (\ref{G1}) in the first and the second inequalities, (\ref{S1}) and Young's inequality in the third and fourth inequalities, respectively.
$C_s, \, \delta^{\prime}, \, \delta, \, C(\delta^{\prime}), \, C(\delta)$ are all positive constants independent of $K$.
In the last inequality, we used the following arguments that are first shown in \cite{Rui}:
\begin{equation}\label{II} I \leq 2 \sum_{m=1}^K\, ||m^q h_m||_{H_{\Lambda}^s}^2 \cdot \sum_{n=1}^K\, ||n^q h_n||_{H_{x,v}^s}^2\,, \qquad
II \leq c\, \sum_{k=1}^K\, ||k^q f_k||_{\Lambda}^2\,, \end{equation}
with $c$ a constant independent of $K$.
To get (\ref{II}), one writes $$I=\sum_{m=1}^{K}\, ||m^q h_m||_{H_{\Lambda}^s}^2\, I_m, \qquad
I_m= \sum_{n,k=1}^{K}\, n^{p-q}\, ||n^{q}h_n||_{H^s_{x,v}}^2\, \chi_{mnk}\,.$$
By definition (\ref{chi}) and (\ref{mnk1}), $\chi_{mnk}=1$ indicates that $m-n\leq k\leq m+n$ by (\ref{mnk}), so $I_{m}$ has at most $2n$ choices for a fixed $n$. That is,
$$I_m\leq 2\sum_{n=1}^{K}\, n^{p-q+1}\, ||n^q h_n||_{H^s_{x,v}}^2 \leq 2 \sum_{n=1}^{K}\,  ||n^q h_n||_{H^s_{x,v}}^2\,,$$
if $q>p+1$. Similarly, $$II\leq 2\sum_{m=1}^{K}\, n^{p-q+1}\, \sum_{k=1}^{K}\, ||k^q f_k||_{\Lambda}^2\leq c\, \sum_{k=1}^K\, ||k^q f_k||_{\Lambda}^2\,,$$
since for each fixed pair $(n,k)$, there are at most $2n$ choices for $m$. If $q>p+2$, $c=2\sum_{n=1}^{\infty} n^{p-q+1}
\leq 2 (1+(p-q+2)^{-1})$. \,
For the terms with $m\leq n$, one exchanges the indexes $m$ and $n$ and can get the same conclusion. Thus we have validated Assumption ${\bf H6}$
for the Boltzmann equation. \\[2pt]

To get Theorem \ref{thm3}, we also need to take care of the linearized operator $\mathcal L$, which is an analog to the proof of Theorem \ref{thm2}.
For simplicity, we show the case $s=1$ with the energy estimate on $||h_k||_{\mathcal H_{\epsilon_{\perp}}^1}$, defined by
$$ ||h_k||_{\mathcal H_{\epsilon_{\perp}}^1}^2 = A\, ||h_k||_{L^2_{x,v}}^2 + \alpha\, ||\nabla_x h_k||_{L^2_{x,v}}^2 + b\, ||\nabla_v h_k^{\perp}||_{L^2_{x,v}}^2
+ a\, \epsilon \langle\nabla_x h_k, \,\nabla_v h_k\rangle_{L^2_{x,v}}\,.$$
For higher order Sobolev norm, one can refer to \cite{MB} and easily extend the conclusion.

When estimating $||h_k||_{L^2_{x,v}}^2$ for $k=1, \cdots, K$,  multiplying $k^{2q}\, h_k$ to both sides of (\ref{h_gPC1}) and integrating on $x, v$, then the term involving $\mathcal L$ is
\begin{equation}\label{L1}\frac{1}{\epsilon^2}\, \sum_{k=1}^{K}\, k^{2q}\, \langle\mathcal L(\sum_{i=1}^{K} \widetilde S_{ki}h_i), \, h_k\rangle_{L^2_{x,v}}\,. \end{equation}
Since the collision kernel is assumed to be linear in $z$, similar to the argument in (\ref{mnk}), $\widetilde S_{ki}$ may be nonzero only when $i$ has three choices:
$$i=k-1, \, k, \, k+1, $$
or $k=i-1, \, i, \, i+1$, then \begin{equation}\label{ki}\frac{k}{2}\leq \frac{i+1}{2}\leq i. \end{equation}
Under the assumptions given in Theorem \ref{thm3}, one can see that $|\widetilde S_{kk}|\leq C_b$ and
\begin{equation}\label{ski} |\widetilde S_{ki}| \leq \xi\, ||z||_{L^{\infty}}\langle |\psi_k|, |\psi_i|\rangle_{L_z^2}\leq \xi\, C_z\, ||\psi_k||_{L^2_z}\, ||\psi_i||_{L^2_z}
=C_2(\xi), \qquad\text{if}\quad  k\neq i,  \end{equation}
where the constant $C_2(\xi)=\xi\, C_z= O(\epsilon)$.

If $i=k$, since $\widetilde S_{kk}=b$ and by using the coercivity property (\ref{coercivity}) and integrating on $x$,
one has $$\frac{1}{\epsilon^2}\, \sum_{k=1}^{K}\, k^{2q}\, \langle\mathcal L(h_k), h_k\rangle_{L^2_{x,v}}
\leq -\frac{\lambda}{\epsilon^2}\, \sum_{k=1}^{K}\,  ||k^q\, h_k^{\perp}||_{\Lambda}^2\,,$$
If $i\neq k$, then $i=k-1$ or $i=k+1$. Define
\begin{align}
&\label{chi1}\displaystyle \chi_{ki} = \begin{cases} 0, \qquad \widetilde S_{ki}=0, \\[2pt]
 1, \qquad\widetilde S_{ki}\neq 0,  \end{cases}
\end{align}
and use (\ref{LL}) to bound (\ref{L1}) by some positive terms,
\begin{align}
&\displaystyle\quad\frac{1}{\epsilon^2}\, \sum_{k=1}^{K}\, k^{2q}\, \langle \mathcal L(\sum_{i=1}^{K} \widetilde S_{ki} h_i), \, h_k\rangle_{L^2_{x,v}} \leq \frac{C^{\mathcal L}}{\epsilon^2}\, \sum_{k,i=1}^{K} k^{2q}\, ||\widetilde S_{ki}\, h_i||_{\Lambda}\, ||h_k||_{\Lambda}\notag\\[2pt]
&\displaystyle \leq \frac{C^{\mathcal L} C_2(\xi)}{\epsilon^2}\, \sum_{k,i=1}^{K} \frac{k^{2q}}{i^q}\, ||i^q h_i||_{\Lambda}\, ||h_k||_{\Lambda}
\leq \frac{C^{\mathcal L} C_2(\xi)}{\epsilon^2}\, 2^q\, \sum_{k,i=1}^{K}  ||i^q h_i||_{\Lambda}\, ||k^q h_k||_{\Lambda}\, \chi_{ki}\notag\\[2pt]
&\displaystyle \leq \frac{C^{\mathcal L} C_2(\xi) C_3(q)}{\epsilon^2}\, \sum_{k,i=1}^{K}  ||i^q h_i||_{\Lambda}\, \chi_{ki} + \frac{C^{\mathcal L} C_2(\xi)C_3(q)}{\epsilon^2}\, \sum_{k,i=1}^{K} ||k^q h_k||_{\Lambda}^2\, \chi_{ki}\notag\\[2pt]
&\displaystyle \label{LK}\leq \frac{2 C_1}{\epsilon^2}\, \sum_{i=1}^{K} ||i^q h_i||_{\Lambda}^2 + \frac{2 C_1}{\epsilon^2}\, \sum_{k=1}^{K} ||k^q h_k||_{\Lambda}^2
= \frac{4C_1}{\epsilon^2}\, \sum_{i=1}^{K} ||i^q h_i||_{\Lambda}^2\,,
\end{align}
where we use $\displaystyle\frac{k^{2q}}{i^q}\leq 2^q k^q$ by (\ref{ki}) and Young's inequality in the third and fourth inequalities, respectively. The fifth inequality is because $i,k$ have only 2 choices for a fixed $k,i$, respectively. $C_3(q)$ is a constant depending on $q$ and we denote the constant
$C_1=C^{\mathcal L}\, C(\xi)\,  C_3(q)=\mathcal O(\epsilon)$.

One observes that $C_1$ in the nominator can be cancelled with an $\epsilon$ in the denominator on the right-hand-side of (\ref{LK}), then the whole term is of $\mathcal O(1/\epsilon)$.  Note from (\ref{h_gPC1}) that the nonlinear term $\Gamma$ has coefficient $1/\epsilon$.
We combine (\ref{LK}) with $1/\epsilon$ multiplied to the second term on the right-hand-side of (\ref{H6_1}), wherein $f_k=h_k$ in (\ref{H6_1}) since $\sum_{k=1}^{K}\, ||k^q h_k||_{\Lambda}^2$
is estimated, then the rest of the proof is the same as that in Theorem \ref{thm2}.
The same estimate holds if substituting $h_k$ by $\nabla_x h_k$ in (\ref{LK}).

Let $C_4=\max(C_b, C_2(\xi))$. To get an estimate of $||\nabla_v h_k^{\perp}||_{L^2_{x,v}}^2$, by using (\ref{H1_1}), (\ref{H1}) and (\ref{H2}) in Assumption ${\bf H1}$ and ${\bf H2}$, one gets the following term involving $\mathcal L$:
\begin{align}
&\displaystyle \quad \frac{1}{\epsilon^2}\sum_{k=1}^{K}\, k^{2q}\, \langle\, \nabla_v\mathcal L(\sum_{i=1}^{K}\widetilde S_{ki}\, h_i^{\perp}), \, \nabla_v h_k^{\perp}\rangle_{L^2_{x,v}}
=\frac{1}{\epsilon^2} \sum_{k,i=1}^{K}\, k^{2q}\, \langle\, \nabla_v\mathcal L(\widetilde S_{ki}\, h_i^{\perp}), \, \nabla_v h_k^{\perp}\rangle_{L^2_{x,v}}\notag \notag\\[2pt]
&\displaystyle \leq \frac{C_4}{\epsilon^2}\sum_{k,i=1}^{K}\, k^{2q} \left( (C(\delta)\frac{\nu_1^{\Lambda}}{\nu_0^{\Lambda}} + \nu_4^{\Lambda})\, ||h_i^{\perp}||_{\Lambda}^2 +
(\delta\frac{\nu_1^{\Lambda}}{\nu_0^{\Lambda}}-\nu_3^{\Lambda})\, ||\nabla_v h_k^{\perp}||_{\Lambda}^2  \right) \notag \\[2pt]
&\displaystyle = \frac{C_4}{\epsilon^2}\sum_{k,i=1}^{K}\, \frac{k^{2q}}{i^{2q}} (C(\delta)\frac{\nu_1^{\Lambda}}{\nu_0^{\Lambda}} + \nu_4^{\Lambda})\, ||i^q\, h_i^{\perp}||_{\Lambda}^2 + \frac{C_4}{\epsilon^2}\sum_{k,i=1}^{K}\, (\delta\frac{\nu_1^{\Lambda}}{\nu_0^{\Lambda}}-\nu_3^{\Lambda})\, ||k^q\, \nabla_v h_k^{\perp}||_{\Lambda}^2 \notag\\[2pt]
&\displaystyle\label{LKV} \leq \frac{3\times 4^{q}\, C_4}{\epsilon^2}\, (C(\delta)\frac{\nu_1^{\Lambda}}{\nu_0^{\Lambda}} + \nu_4^{\Lambda}) \sum_{i=1}^{K} ||i^q\,  h_i^{\perp}||_{\Lambda}^2 + \frac{3 C_4}{\epsilon^2}\, (\delta\frac{\nu_1^{\Lambda}}{\nu_0^{\Lambda}}-\nu_3^{\Lambda})\sum_{k=1}^{K} ||k^q \, \nabla_v h_k^{\perp}||_{\Lambda}^2\,,
\end{align}
where we use (\ref{ki}) in the second equality, and the fact that both $k, i$ have only 3 choices for a fixed $i, k$, respectively. (\ref{LKV}) is similar to
the estimate $||\nabla_v h^{\perp}||_{L^2_{x,v}}$ used in the proof of Theorem \ref{thm2}, in the sense that $h$ is substituted by $k^q\, h_k$ and summing up $k=1, \cdots, K$.
The estimate for $\langle \nabla_x h_k, \nabla_v h_k\rangle_{L^2_{x,v}}$ is also similar in the above way to the estimate $\langle\nabla_x h,
\nabla_v h\rangle_{L^2_{x,v}}$ when proving Theorem \ref{thm2} in \cite{MB}.

Analogous to the proof of Theorem \ref{thm2}, we multiply by $k^{2q}$ on both sides of all the estimates and sum up $k=1, \cdots, K$,
then achieve the same result as Theorem \ref{thm2}, that is, the exponential decay of \, $\sum_{k=1}^{K}\, ||k^q h_k||_{H_{x,v}^s}^2$.
Therefore, Theorem \ref{thm3} is proved.    \\[1pt]

As a corollary, with the assumption (\ref{basis}) on the basis, using the Cauchy-Schwarz inequality and the norm definition (\ref{h_sup}), 
one gets the estimate for the gPC solution $h^K$, 
\begin{align}
&\displaystyle ||h^K||_{H_{x,v}^s L_z^{\infty}}^2 = ||\sum_{k=1}^{K} h_k\, \psi_k(z)||_{H_{x,v}^s L_z^{\infty}}^2
\leq C\, \sum_{k=1}^{K}\, k^{2p}\, ||h_k||_{H_{x,v}^s}^2 \notag\\[2pt]
&\label{h_K1}\displaystyle\qquad\qquad\quad\leq C\, \bigg(\sum_{k=1}^{K} k^{q}\, ||h_k||_{H_{x,v}^s}^2\bigg)\bigg(\sum_{k=1}^{K}\, k^{2(p-q)}\bigg)
\leq C^{\prime}\,\sum_{k=1}^{K}\, ||k^q h_k||_{H_{x,v}^s}^2,
\end{align}
since $q>p+2$, $C$ is shown in (\ref{basis}) and $C^{\prime}$ is a constant independent of $K$.
Therefore, $\displaystyle||h^K||_{H_{x,v}^s L_z^{\infty}}$ also decays exponentially in time with the same rate as $E^K(t)$, namely
\begin{equation}\label{h_K2}||h^K||_{H^s_{x,v}L_z^{\infty}} \leq \tilde\eta\, e^{-\tau t}
\end{equation}
in the incompressible Navier-Stokes scaling, and
\begin{equation}\label{h_K3} ||h^K||_{H^s_{x,v}L_z^{\infty}} \leq \tilde\eta\, e^{-\epsilon\tau t} 
\end{equation} in the acoustic scaling. Here $\tilde\eta = C^{\prime}\eta$. 
Define the norm for general function $g$, 
\begin{equation}\label{h_e} ||g||_{H_{x,v}^{s} L_z^2}^2 := \int_{I_z}\, ||g||_{H_{x,v}^s}^2 \pi(z)\, dz. 
\end{equation}
Then (\ref{h_K2}) and (\ref{h_K3}) further implies that 
\begin{corollary}
\label{Col}
Suppose all the assumptions in Theorem \ref{thm3} are satisfied, we have the following estimates for the gPC solution $h^K$: \\
(i) Under the incompressible Navier-Stokes scaling, 
\begin{equation*}
||h^K||_{H^s_{x,v}L_z^{\infty}} \leq \tilde\eta\, e^{-\tau t}, \qquad ||h^K||_{H^s_{x,v}L_z^2} \leq \tilde\eta\, e^{-\tau t}\,.  
\end{equation*}
(ii) Under the acoustic scaling, 
\begin{equation*}
||h^K||_{H^s_{x,v}L_z^{\infty}} \leq \tilde\eta\, e^{-\epsilon\tau t}, \qquad ||h^K||_{H^s_{x,v}L_z^2} \leq \tilde\eta\,  e^{-\epsilon\tau t}\,,  
\end{equation*}
where $\tilde\eta$, $\tau$ are positive constants independent of $K$ and $z$. 
\end{corollary}
For other kinetic models like the Landau equation, the proof is similar and we omit it here.

\subsection{Estimate of the gPC Error}
\label{subsec:gPC}
We first give the main result of this section on the estimate of the gPC error:
\begin{theorem}
\label{thm4}
Suppose the assumptions on the collision kernel and basis functions in Theorem \ref{thm3}, and the assumption for initial data 
$||h_{in}||_{H_{x,v}^{s,r} L_z^{\infty}} \leq C_{I}$ are satisfied, we then have the following: \\
(i)\, Under the incompressible Navier-Stokes scaling,
\begin{equation}\label{Thm4_1}  ||h^e||_{H_{x,v}^{s} L_z^2} \leq C_{e}\, \frac{e^{-\lambda t}}{K^r}\,. \end{equation}
(ii)\, Under the acoustic scaling, 
\begin{equation}\label{Thm4_2} ||h^e||_{H_{x,v}^{s} L_z^2} \leq C_{e}\, \frac{e^{-\epsilon\lambda t}}{K^r}\,, \end{equation}
where the constants $C_{e}, \,\lambda>0$ are independent of $K$ and $\epsilon$.
\end{theorem}

Recall the reconstructed gPC solution defined in (\ref{ans}) and the projection operator in (\ref{proj}). The total gPC error is given by
\begin{equation}\label{Total_err}h^{e}=h-h^{K}:=\underbrace{h-P_{K} h}_{R^K}\, + \underbrace{P_{K} h-h^K}_{e^K}\,, \end{equation}
where $R^K$ is the projection error and $e^K$ is the numerical error.

By Theorem \ref{thm2} and the standard estimate on the projection error,
\begin{equation}\label{PE}
||R^K||_{H_{x,v}^s L_z^2}\leq ||R^K||_{H_{x,v}^s L_z^{\infty}} \leq C_{P}\, \frac{||h||_{H_{x,v}^{s} H_z^{r+1}}}{K^r} \leq C_{P}\, C_{I}\,
\frac{e^{-\tau_s t}}{K^r}\,,\end{equation}
where $C_P$, $C_I$ are constants. 
Define 
\begin{equation}\label{eK_Def} e^K := P_{K} h -h^K = \sum_{k=1}^{K}\, (\hat h_k(t,x,v)-h_k(t,x,v))\, \psi_k(z):=\sum_{k=1}^{K}\, e_k(t,x,v)\, \psi_k(z)\,, 
\end{equation} where one defines the coefficients of $e^K$ by
$$ e_k = \hat h_k -h_k, \qquad 1\leq k\leq K, \qquad {\bf e}=(e_1, \cdots, e_K)^{T}. $$

We discuss the incompressible Navier-Stokes scaling below.
In order to get the hypocoercivity estimate for $h^e$, one needs to write down the gPC system for the numerical error $e_k$
\begin{align}
\label{e_k}
\left\{
\begin{array}{l}
\displaystyle \partial_t e_k +  \frac{1}{\epsilon}\, v\cdot\nabla_x e_k =\frac{1}{\epsilon^2}\left(\mathcal L_{k}(h) -\mathcal L_{k}(h^K)\right) +
\frac{1}{\epsilon}( \mathcal F_{k}(h,\,h) -\mathcal F_{k}(h^K, h^K))\,,  \\[4pt]
\displaystyle e_{k}(0,x,v)=0,  \qquad x\in\Omega\subset\mathbb T^d,  \, v\in \mathbb R^d.
\end{array}\right.
\end{align}
Since 
$$\mathcal L_{k}(h)-\mathcal L_{k}(h^K) = \mathcal L_{k}(h-P_{K}h+P_{K}h-h^K) = \mathcal L_{k}(e^K) + \mathcal L_{k}(R^K). $$
By the estimate (\ref{PE}), the expression for $\mathcal L_{k}$ given in (\ref{LFK}), 
we observe that the term 
involving $S_{kk}$ in $\mathcal L_k(R^K)$ equals to zero because of the assumption that leading order of random collision kernel 
being constant and definition of $R^K$; and the term involving $S_{ki}$ for $i=k-1$ or $i=k+1$ contributes to an order $\epsilon$ coefficient thanks to the
small random perturbation assumption (\ref{b1}) of the collision kernel, thus the estimate for $\mathcal L_{k}(R^K)$ is given by
\begin{equation}\label{LRK} ||\mathcal L_{k}(R^K)||_{H_{x,v}^s L_z^2} \leq C\, \epsilon\,  \frac{e^{-\tau_s t}}{K^r}, 
\end{equation}
where $C$ is a constant independent of $K$ and $z$. 

Based on Corollary 1 to Theorem 3 shown in \cite{Lu} (see also \cite{BD}), if $B$ satisfies the assumption (\ref{BK1}) on
the uniform boundness of  $z$-derivatives of the collision kernel, then for any $f, g \in L_v^1\cap L_v^2$\,,
$$||\mathcal Q(f,\, g)||_{H^s_{x,v}}\leq C_{\text{ker}}\, ||f||_{L^2_{x,v}}\, ||g||_{L^2_{x,v}}\,,$$
where $C_{\text{ker}}>0$ depending only on the collision kernel and is independent of $z$.
Using (\ref{GB}), one gets
\begin{equation}\label{F_Est}
||\mathcal F(f,\,g)||_{H_{x,v}^{s}} \leq C_{\text{ker}}\, ||f||_{L^2_{x,v}}\, ||g||_{L^2_{x,v}} \leq C_{\text{ker}}\, ||f||_{H_{x,v}^{s}}\, ||g||_{H_{x,v}^{s}}\,. 
\end{equation}

Notice that $e_k$ satisfies the similar estimate as (\ref{h_perp}) in section \ref{Proof_Thms}, except for the nonlinear terms involving the 
operator $\mathcal F$, whose estimate is given by (\ref{App_R2}) in the Appendix. 
Repeating the similar procedure as the proof for Theorem \ref{thm2}, using $ ||e^K||_{\mathcal H_{\epsilon_{\perp}}^s L_z^2}^2 
= \sum_{k=1}^{K} ||e_k||_{\mathcal H_{\epsilon_{\perp}}^s}^2$,  
and the estimate (\ref{LRK}) multiplying by the coefficient $1/\epsilon^2$, which is of order $1/\epsilon$ thus can be combined with the 
nonlinear estimates and is shown in the third term on the RHS of (\ref{e1}) below, 
one gets  
\begin{equation}
\label{e1}
\frac{1}{2}\frac{d}{dt}||e^K||_{\mathcal H_{\epsilon_{\perp}}^s L_z^2}^2
 \leq\left( -A_1 ||e^K||_{H^s_{x,v}L_z^2} + A_2\, e^{-\tau_1 t}\,
||e^K||_{H^s_{x,v}L_z^2} + A_3\, \frac{e^{-\tau_2 t}}{K^{r}}\right) ||e^K||_{H^s_{x,v}L_z^2}\,, \end{equation}
where
$\tau_2 = \tau_1 + \tau_s$, and $A_i>0$ ($i=1,2,3$) coming from $C_I$, $C_P$, $C_{\text{ker}}$
are all constants independent of $K$ and $\epsilon$. 
Choose the initial data $h_{in}$ with $||h_{in}||_{H_{x,v}^{s,r}L_z^{\infty}}\leq C_I$ such that $A_2<A_1$.
Since $||\cdot||_{\mathcal H_{\epsilon_{\perp}}^s}$ is equivalent to the $||\cdot||_{H^s_{x,v}}$ norm, thus
\begin{equation}\label{e2} 
\frac{d}{dt} ||e^K||_{H^s_{x,v}L_z^2} \leq - A_4\,  ||e^K||_{H^s_{x,v}L_z^2} + \widetilde A_3\, \frac{e^{-\tau_2 t}}{K^{r}}\,. 
\end{equation}

One usually chooses initial data such that ${\bf e}={\bf 0}$, then a simple Gronwall's inequality argument gives
\begin{equation}\label{RE} ||e^K||_{H^s_{x,v}L_z^2} \leq A_5\, \frac{e^{-\kappa t}}{K^{r}}\,,   \end{equation}
where $\widetilde A_3, \, A_4, \, A_5, \, \kappa>0$ are all constants, independent of $K$ and $\epsilon$. 
Combining (\ref{PE}) and (\ref{RE}), since $h^e=R^K + e^K$, one then has
$$ ||h^e||_{H_{x,v}^{s} L_z^2} \leq C_{e}\, \frac{e^{-\lambda t}}{K^r}\,. $$
where $C_e, \, \lambda>0$ are constants independent of $K$ and $\epsilon$. 
For the acoustic scaling, we multiply by $\epsilon$ on the right-hand-side of
(\ref{e1}) and (\ref{e2}) then obtain (\ref{Thm4_2}). Thus Theorem \ref{thm4} is proved.  \\[1pt]

We would like to point out that the above proof is different from \cite{Hu}. 
Our analysis gives sharper constants which are independent of $K$ and $\epsilon$. 
Theorem \ref{thm4} shows that the gPC-SG method for the Boltzmann equation with random inputs and both scalings is of spectral accuracy.
In addition, the total gPC error $h^e$ decays exponentially in time.

\begin{remark}
As all of our estimates were established in weighted $L_z^2$ or corresponding Sobolev norms, 
our time-decay results (for both analytic and numerical solutions) will hold for both expected value and variance of $h$. 
\end{remark}
\section{Conclusion}
\label{conclusion}
In this paper, we first give an exponential decay to the global equilibrium for both linear and nonlinear kinetic models with 
random uncertainties in the initial data and collision kernel, and with small
scales corresponding to both the incompressible Navier-Stokes and the Euler (acoustic)
scalings, using the theoretical framework developed in \cite{CN, MB} for deterministic problems. As an example we obtain the results for the Boltzmann equation, while similar results can also be obtained for other (random) linear and nonlinear kinetic equations whose deterministic counterparts
are covered in \cite{MB}.
Furthermore, for small random perturbation of the collision kernel, 
we prove the exponential time decay of the
gPC-SG solution, the spectral accuracy of the gPC-SG method as well as the exponential time decay of the numerical error, under some mild conditions on the orthogonal polynomials.

There remain many interesting questions that desire further research.
For example, whether one can establish a similar analysis for more general
orthogonal polynomials not satisfying condition (\ref{basis}), random variable
$z$ in unbounded domain, Cauchy problem in $\mathbb R^d$, and uncertainties
arising from boundary conditions for boundary value problems.


\appendix
\numberwithin{equation}{section}
\section{Proof of the estimate used in Theorem \ref{thm4}}

Take $\partial_l^j$ derivatives (with $|j|+|l|\leq s$) on equation (\ref{e_k}), multiply by $\partial_l^j e_k$ 
and sum up all the $k=1, \cdots, K$, then integrate on $x$ and $v$. 
Denote $dz_1 = \pi(z)dz$ and $d\mu = dxdv$,  the nonlinear term on the right-hand-side is given by 
\begin{align}
&\displaystyle \text{RHS} = \sum_{k=1}^{K} \int \partial_l^j \left(\mathcal F_{k}(h, h) -\mathcal F_{k}(h^K, h^K)\right)\, \partial_l^j e_k\, d\mu  \notag \\[4pt]
&\displaystyle \qquad = \sum_{k=1}^{K} \int \int\partial_l^j \left(\mathcal F(h, h) - \mathcal F(h^K, h^K)\right) \psi_k(z) dz_1\, \partial_l^j e_k\, d\mu  \notag\\[4pt]
&\displaystyle \qquad = \int \int \partial_l^j \left(\mathcal F(h, h) - \mathcal F(h^K, h^K)\right)\, \partial_l^j e^{K}\, d\mu dz_1 \notag\\[4pt]
&\displaystyle\label{App_R1} \qquad\leq  \underbrace{\left(\int\int (\partial_l^j \left(\mathcal F(h,\,h) - \mathcal F(h^K, h^K)\right))^2\, d\mu dz_1\right)^{\frac{1}{2}}}_{\textcircled{a}}
\cdot \underbrace{\left( \int \int (\partial_l^j e^K)^2\, d\mu dz_1\right)^{\frac{1}{2}}}_{\textcircled{b}}, 
\end{align}
where we used the definition of $\mathcal F_{k}$ in the first equality, (\ref{eK_Def}) in the second equality and
the Cauchy-Schwarz inequality in the last inequality. It is obvious that $\textcircled{b} \leq ||e^K||_{H_{x,v}^{s}L_z^2}$. Now we estimate term $\textcircled{a}$.

Using the relation
\begin{equation}\label{F_est} \mathcal F(h, h) -\mathcal F(h^K, h^K)= \mathcal F(h-h^K, h) + \mathcal F(h^K, h-h^K), 
\end{equation}
then 
\begin{align}
&\displaystyle \textcircled{a}^2 \leq ||\mathcal F(h, h) - \mathcal F(h^K, h^K)||_{H_{x,v}^{s}L_z^2}^2 \notag\\[4pt]
&\displaystyle\, \leq 2 \left( ||\mathcal F(h-h^K, h)||_{H_{x,v}^{s}L_z^2}^2 + ||\mathcal F(h^K, h-h^K)||_{H_{x,v}^{s}L_z^2}^2\right) \notag\\[4pt]
& \label{Term_a}\displaystyle\, \leq 2\, C_{\text{ker}}^2\, ||h-h^K||_{H_{x,v}^{s}L_z^2}^2 \left( ||h||_{H_{x,v}^{s}L_z^2}^2 + ||h^K||_{H_{x,v}^{s}L_z^2}^2 \right). 
\end{align}
By Theorem \ref{random_kernel} and Corollary \ref{Col}, we know
\begin{equation}\label{a1} ||h^K|| _{H_{x,v}^{s}L_z^2} \leq \tilde\eta\, e^{-\tau t}, \qquad ||h||_{H_{x,v}^{s}L_z^2} \leq C_{I}\, e^{-\tau_s t}. 
\end{equation}
Using (\ref{Total_err}) and the estimate of $R^K$ in (\ref{PE}), one has
\begin{equation}\label{a2} || h - h^K || _{H_{x,v}^{s}L_z^2}^2 \leq 2 \left( ||R^K||_{H_{x,v}^{s}L_z^2}^2 + ||e^K||_{H_{x,v}^{s}L_z^2}^2\right) 
\leq 2 \left( (C_P C_I)^2\, \frac{e^{-2 \tau_s t}}{K^{2r}} + ||e^K||_{H_{x,v}^{s}L_z^2}^2 \right).  \end{equation}

Plug in (\ref{a1}) and (\ref{a2}) into (\ref{Term_a}), then 
$$ \textcircled{a}^2 \leq (C^{\ast})^2\, e^{-2 \tau_1 t} \left( (C_P C_I)^2\, \frac{e^{-2 \tau_s t}}{K^{2r}} + ||e^K||_{H_{x,v}^{s}L_z^2}^2\right), $$
which gives 
$$\textcircled{a} \leq C^{\ast}\, e^{-\tau_1 t} \left( C_P C_I\, \frac{e^{-\tau_s t}}{K^r} +  ||e^K||_{H_{x,v}^{s}L_z^2}\right), $$
where $C^{\ast} = 2\sqrt{2}\, C_{\text{ker}}\max\{\tilde\eta, C_{I}\}$. 
With the estimate of term $\textcircled{a}$, $\textcircled{b}$ in (\ref{App_R1}), we conclude that 
\begin{equation}\label{App_R2} \text{RHS} \leq C^{\ast}\, e^{-\tau_1 t} \left( C_P C_I\, \frac{e^{-\tau_s t}}{K^r} +  ||e^K||_{H_{x,v}^{s}L_z^2} \right)\, ||e^K||_{H_{x,v}^{s}L_z^2}\,. 
\end{equation}

\bibliographystyle{siam}
\bibliography{main.bib}
\end{document}